\documentclass[11pt]{article}
\usepackage{amsmath,amstext,amssymb,amsfonts,amsthm}
\usepackage{mathrsfs}
\usepackage{graphicx}
\usepackage{multirow}
\usepackage{diagbox}
\usepackage{bm}
\newtheorem{lem}{Lemma}[section]

\newtheorem{thm}{Theorem}[section]

\newtheorem{rem}{Remark}[section]
\newtheorem{exmp}{Example}

\newtheorem{alg}{Algorithm}[section]
\numberwithin{equation}{section}

\newcommand{\D}{\mathbb{D}}

\newcommand{\A}{\mathbb{A}}
\newcommand{\X}{\mathbb{X}}

\newcommand{\G}{\mathbb{G}}
\newcommand{\R}{\mathbb{R}}
\newcommand{\C}{\mathbb{C}}

\newcommand{\cM}{\mathcal{M}}

\newcommand{\Ga}{\Gamma}
\newcommand{\GaK}{\Gamma_{\!\!K}}
\newcommand{\Om}{\Omega}
\newcommand{\I}{\mathbb{I}}
\newcommand{\caE}{\mathcal{E}}

\newcommand{\lam}{\lambda}
\newcommand{\bfx}{\mathbf{x}}

\newcommand{\bY}{\mathbf{Y}}
\newcommand{\bW}{\mathbf{W}}

\newcommand{\bB}{\mathbf{B}}
\newcommand{\wbY}{\widetilde{\mathbf{Y}}}

\newcommand{\bR}{\mathbf{R}}

\newcommand{\bV}{\mathbf{V}}
\newcommand{\bv}{\mathbf{v}}
\newcommand{\bw}{\mathbf{w}}

\newcommand{\bfa}{\mathbf{a}}
\newcommand{\bfb}{\mathbf{b}}
\newcommand{\bfq}{\mathbf{q}}

\newcommand{\bmp}{\bm{\psi}}
\newcommand{\E}{\mathbb{E}}

\newcommand{\fespaceq}{\mathbf{W}_p(\mathcal{M})}
\newcommand{\cam}{\mathcal{M}}
\newcommand{\bM}{\mathbb{M}}
\newcommand{\cP}{\mathcal{P}}

\renewcommand{\Re}{{\rm Re}}
\renewcommand{\i}{{\bf i}}

\newcommand{\lj}{[{\hskip -1.5pt} [}
\newcommand{\rj}{]{\hskip -1.5pt} ]}

\newcommand{\la}{\langle}
\newcommand{\ra}{\rangle}

\newcommand{\pa}{\partial}

\newcommand{\be}{\begin{eqnarray}}
\newcommand{\ee}{\end{eqnarray}}
\newcommand{\beq}{\begin{equation}}
\newcommand{\eeq}{\end{equation}}
\newcommand{\ben}{\begin{eqnarray*}}
\newcommand{\een}{\end{eqnarray*}}
\newcommand{\nn}{\nonumber}

\allowdisplaybreaks[4]

\title{Efficient and Parallel Solution of High-order Continuous Time Galerkin for Dissipative and Wave Propagation Problems\footnotemark[1]}
\author{
Zhiming Chen\footnotemark[2]
\and
Yong Liu\footnotemark[3]
}

\date{}

\begin{document}
\maketitle

\renewcommand{\thefootnote}{\fnsymbol{footnote}}
\footnotetext[1]{This work is supported in part by China National Key Technologies R\&D Program under the grant 2019YFA0709600 and China NSF under the grant 118311061, 12288201, 12201621. }
\footnotetext[2]{LSEC, Institute of Computational Mathematics,
Academy of Mathematics and System Sciences and School of Mathematical Science, University of
Chinese Academy of Sciences, Chinese Academy of Sciences,
Beijing 100190, China. E-mail: zmchen@lsec.cc.ac.cn}
\footnotetext[3]{LSEC, Institute of Computational Mathematics, Academy of Mathematics and Systems Science, Chinese Academy of Sciences, Beijing 100190, P.R. China.  E-mail: yongliu@lsec.cc.ac.cn}

\begin{center}
\small
\begin{minipage}{0.9\textwidth}
\textbf{Abstract.}
We propose efficient and parallel algorithms for the implementation of the high-order continuous time Galerkin method for dissipative and wave propagation problems. By using Legendre polynomials as shape functions, we obtain a special structure of the stiffness matrix which allows us to extend the diagonal Pad\'e approximation to solve ordinary differential equations with source terms. The unconditional stability, $hp$ error estimates, and $hp$ superconvergence at the nodes of the continuous time Galerkin method are proved. Numerical examples confirm our theoretical results.

\medskip
\textbf{Key words.}
Implicit time discretization; Pad\'e approximation; Parallel implementation
\medskip

\textbf{AMS classification}.
65M60
\end{minipage}
\end{center}
\setlength{\parindent}{2em}
\section{Introduction}
\label{sec_intro}
In this paper, we study the following system of ordinary differential equations (ODEs)
\begin{align}\label{p1}
{\bY}'(t)=\D\bY(t)+\bR(t) \text{ in } (0,T), \ \ \ \ \bY(0)=\bY_0,
\end{align}
which is obtained from the method-of-lines approach for linear partial differential equations (PDEs) after space discretization. Here $T>0$ is the length of the time interval, $\bY,\bR\in \mathbb{R}^M$, and $\D$ is an $M\times M$ real constant matrix, where $M$ is the number of degrees of freedom of the spatial discretization. Without loss of generality, we assume
\begin{align}\label{p2}
\D+\D^T\leq 0,
\end{align}
that is, $\D+\D^T$ is a semi-negative definite matrix.  This condition is satisfied by a large class of linear PDEs including the dissipative problems such as the parabolic equations and the wave propagation problems such as the wave equation and Maxwell equations.

Let $0=t_0<t_1<\cdots<t_N=T$ be a partition of $(0,T)$. If the source $\bR=\bf{0}$ in \eqref{p1}, the exact solution in each time interval $(t_n,t_{n+1})$ is $\bY(t)=e^{\D(t-t_n)}\bY(t_n)$ for which Pad\'e approximation to the exponential function can be used to construct and analyze numerical schemes to solve \eqref{p1}. In \cite{Saad1989}, by using the partial fraction formula for the Pad\'e approximation, the $[r/r]$, $r\ge 1$, Pad\'e approximation leads to the following method
\beq\label{a1}
\bY(t_{n+1})\approx\frac{P_r(\tau_n\D)}{P_r(-\tau_n\D)}\bY(t_n)=\left[(-1)^r\I+\sum^r_{j=1}-\frac{P_r(-\zeta_j)}{P'_r(\zeta_j)}(\zeta_j\I+\tau_n\D)^{-1}\right]\bY(t_n),
\eeq
where $\tau_n=t_{n+1}-t_n$, $\I\in\R^{M\times M}$ is the identity matrix, $P_r(z)$ is the numerator of the $[r/r]$ Pad\'e approximation to the exponential function $e^z$, and $\{\zeta_1,\cdots,\zeta_r\}$ are zeros of $P_r(z)$ which are known to be simple and lie in the left-half plane. \eqref{a1} indicates that one can compute the approximation of the solution $\bY(t_{n+1})$ in each time step by solving $k$ complex matrix problems and $r-2k$ real matrix problems of the form $\zeta\I+\tau_n\D$ in parallel, where $k$, $0\le k\le r/2$, is the number of complex zeros of $P_r(z)$ (see Remark \ref{rem1} below). The purpose of this paper is to construct algorithms sharing this very desirable property for solving \eqref{p1} with general nonzero sources $\bR(t)$.

There exists a large literature on implicit single-step time-stepping methods for solving \eqref{p1} (see, e.g.,\cite{Hairer1996}
and the references therein). The following continuous time Galerkin method proposed in \cite{Hulme} is probably the simplest
\beq\label{a2}
\bY'_r=\D\cP_{r-1}\bY_r+\cP_{r-1}\bR \ \ \ \mbox{in }(t_n,t_{n+1}),\ \ \ 0\le n\le N-1,
\eeq
where $\bY_r$ is a piecewise polynomial of degree $r\ge 1$ in each interval $(t_n,t_{n+1})$, continuous at the nodes $t=t_n$, and $\cP_{r-1}$ is the local $L^2$ projection to the space of polynomials of degree $(r-1)$ in each interval. It is shown in \cite{Hulme} that \eqref{a2} is equivalent to the $r$-stage Gauss collocation method at the nodes when $\bR=\bf{0}$ and has the highest classical order $2r$ among all $r$-stage Runge-Kutta methods \cite[Table 5.12]{Hairer1996}. The continuous time Galerkin method, together with finite element discretization in space, is used in \cite{Aziz} for the heat equation and in \cite{French1996CTG}, \cite{Monk2014JSC} for the wave equation. We refer to \cite{Akrivis2011NM} for a unified framework and the comparison of the most popular implicit single-step time-stepping methods including also the discontinuous time Galerkin method and various Runge-Kutta methods.

The difficulty in using the high-order continuous time Galerkin method or any implicit time Runge-Kutta methods is that a straightforward implementation requires to solve a system of linear equations of the size $rM\times rM$, which is not feasible in most time for PDE problems. In a recent work \cite{Southworth}, efficient iterative algorithms are developed based on optimal preconditioning of the stage matrix for finding the stage vectors of the implicit Runge-Kutta methods for solving \eqref{p1}. For an $r$-stage implicit Runge-Kutta method, the stage matrix is an $r\times r$ block matrix with each block being a $M\times M$ matrix. One can find further references in \cite{Southworth} for developing efficient algorithms implementing the high order implicit time discretization methods in the literature. We also refer to \cite{Schwab}, \cite{Richter} for the implementation of the discontinuous time Galerkin method based on the block diagnalization of the stiffness matrix.

In this paper we propose an efficient realization of the method \eqref{a2} which uses Legendre polynomials as shape functions to obtain a new stiffness matrix which is different from the stage matrix in \cite{Southworth} applying to the Gauss collocation method. By exploiting the special structure of the stiffness matrix, we construct an algorithm which computes the solution $\bY_r(t_n)$ at each node by solving $k$ complex matrix problems and $r-2k$ real matrix problems in parallel, where $k$, $0\le k\le r/2$, is the number of complex zeros of the $[r/r]$ Pad\'e numerator $P_r(z)$. Moreover, a parallel-in-time algorithm is proposed to compute the other coefficients of the solution $\bY_r$ in each time interval $(t_n,t_{n+1})$ which solves in parallel $kr$ complex matrix problems and $(r-2k)r$ real matrix problems. For the dissipative system, in which $\D+\D^T$ is negative definite, in the parallel-in-time algorithm, only $k$ complex matrix problems and $(r-2k)$ real matrix problems need to be solved. We remark that our parallel-in-time algorithm is different from the other parallel-in-time algorithms based on domain decomposition or space-time multigrid techniques in the literature (see, e.g., \cite{Gander}).

As a by-product of our analysis, we obtain the following formula (Theorem \ref{thm:3.3}) to compute the nodal values $\bY_r(t_{n+1})$, $0\le n\le N-1$, of the solution of \eqref{a2}
\beq\label{a3}
\bY_r(t_{n+1})=\frac{P_r(\tau_n\D)}{P_r(-\tau_n\D)}\bY_r(t_n)+\sum^r_{k=1}(-1)^{k+1}\frac{\phi_{k1}(\tau_n\D)}{P_r(-\tau_n\D)}\,\bfb_{k-1}+\tau_n\bR_0,
\eeq
where for $k=1,\cdots,r$, $\phi_{k1}(\lam)$ is a polynomial of degree $r$ satisfying some recurrence relations, and $\bfb_k$, $\bR_0$ are vectors depending on the source $\bR$. \eqref{a3} can be viewed as a generalization of the $[r/r]$ Pad\'e approximation \eqref{a1} for solving the ODE system without sources.

The layout of the paper is as follows. In section 2 we introduce the continuous time Galerkin method for \eqref{p1} and prove the strong stability and derive a $hp$ error estimate. In section 3 we propose our parallel algorithms to implement the continuous time Galerkin method. In section 4 we consider an alternative implementation for the dissipative system. In section 5 we prove the optimal stability and error estimates in terms of $r$ when $\D$ is a symmetric or skew-symmetric matrix. In section 6 we consider the application of the algorithms in this paper to solve the linear convection-diffusion equation by using the local discontinuous Galerkin method and the wave equation with discontinuous coefficients by using the unfitted finite element spatial discretization.

\section{Implicit time discretization}\label{sec_implicit}
In this section, we introduce the continuous time Galerkin method for solving \eqref{p1}.
 Let $0=t_0<t_1<\ldots<t_{N}=T$ be a partition of the time interval $(0,T)$ with time steps $\tau_n=t_{n+1}-t_n$, $0\le n\le N-1$. We set
$I_n=(t_{n},t_{n+1})$ and $\tau=\displaystyle\max_{0\le n\le N-1}\{\tau_n\}$.
For any integer $m\geq 1$, we define the finite element space
\begin{align*}
\bV_\tau^{m}:=\{\bv\in [C(0,T)]^M: \bv|_{I_n}\in [P^m]^M,0\leq n \leq N-1\},
\end{align*}
where $P^m$ is the set of polynomials whose degree is at most $m$. Define the local projection $\mathcal{P}_m$, $m\ge 0$, such that in each time interval $I_n$, $\cP_m: [L^2(I_n)]^M\rightarrow [P^m]^M$ satisfies
\begin{align*}
\int_{I_n}(\mathcal{P}_m\bv,\bw)\, dt=\int_{I_n}(\bv,\bw)\,dt \quad \forall \bw \in [P^m]^M,
\end{align*}
where we denote $(\cdot,\cdot)$ the inner product of $\R^M$.
It is well-known (see, e.g., Schwab \cite{Schwab1998hp}) that for any $s\ge 0$, $m\ge 0$,
\begin{align}\label{err}
\|\bv-\cP_m \bv\|_{L^2(I_n)} \leq C\frac{\,\tau_n^{\min(m+1,s)}}{(m+1)^{s}}\|\bv\|_{H^{s}(I_n)}\ \ \forall\bv\in [H^s(I_n)]^M,
\end{align}
where the constant $C$ is independent of $m,\tau_n$ but may depend on $s$. In this paper, for any integer $d\ge 1$ and Banach space $X$, we denote $\|\cdot\|_X$ both the norm of $X$ and $[X]^d$.

For any integer $r\geq 1$, the continuous time Galerkin method for solving \eqref{p1} is to find the function $\bY_r\in \bV_\tau^r$ such that $\bY_r(0)=\bY_0$ and
\begin{align}\label{dp}
\bY_r'=\D\mathcal{P}_{r-1} \bY_r+\mathcal{P}_{r-1} \bR \ \ \mbox{in }I_n,\ \ 0\le n\le N-1.
\end{align}

The following stability lemma extends an idea in Griesmaier and Monk \cite{Monk2014JSC} where the continuous time Galerkin discretization in time and hybridizable discontinuous Galerkin method in space for the wave equation are considered.

\begin{lem}\label{lem:2.2}
The problem \eqref{dp} has a unique solution $\bY_r\in\bV^r_\tau$ which satisfies
\beq
\max_{1\le n\le N}\|\bY_r(t_n)\|_{\R^M}\le \|\bY_0\|_{\R^M}+CT^{1/2}\|\bR\|_{L^2(0,T)},\label{g3}
\eeq
\beq
\max_{0\le t\le T}\|\bY_r\|_{\R^M}\le Cr^2(\|\bY_0\|_{\R^M}+T^{1/2}\|\bR\|_{L^2(0,T)}).\label{g4}
\eeq
where the constant $C$ is independent of $r,\tau,\D$ and $\bR$.
\end{lem}

\begin{proof} At each time step, \eqref{dp} is equivalent to a linear system of equations whose existence and uniqueness of the solution follow from the stability estimate \eqref{g4}. To prove the stability estimates \eqref{g3}-\eqref{g4}, we
denote by $\{L_j\}^\infty_{j=0}$ the Legendre polynomials on $(-1,1)$ and define $\widetilde L_j=L_j\circ\psi^{-1}$, where $\psi:(-1,1)\to (t_n,t_{n+1})$ is the mapping $\psi(\xi)=\frac{t_n+t_{n+1}}2+\frac{t_{n+1}-t_n}2\xi$ for $\xi\in (-1,1)$. Then $\{\widetilde L_j\}^\infty_{j=0}$
are orthogonal in $L^2(I_n)$, $\widetilde L_r(t_n)=(-1)^r, \widetilde L_r(t_{n+1})=1$, and
\beq\label{d2}
\int_{I_n}|\widetilde L_r|^2dt=\frac{\tau_n}{2r+1},\ \ \int_{I_n}|\widetilde L_r'|^2dt=\frac{2r(r+1)}{\tau_n}.
\eeq
For $n=0,\cdots,N-1$, let $\bY_r^n=\bY_r(t_n)$ and $\hat\bY_r\in [P^r]^M$ satisfy
\beq\label{g1}
\hat\bY'_r=\D\cP_{r-1}\hat\bY_r\ \ \mbox{in }I_n,\ \ \hat\bY_r(t_n)=\bY^n_r.
\eeq
By multiplying \eqref{g1} by $\hat\bY_r$ and integrating over $I_n$, we obtain easily by \eqref{p2} that
\ben
\frac 12\|\hat\bY_r(t_{n+1})\|_{\R^M}^2-\frac 12\|\bY^n_r\|_{\R^M}^2=\int_{I_n}(\D\cP_{r-1}\hat\bY_r,\cP_{r-1}\hat\bY_r)dt\le 0.
\een
This implies 
\beq
\|\hat\bY_r(t_{n+1})\|_{\R^M}\le\|\bY^n_r\|_{\R^M}.\label{g5}
\eeq
Since $\hat\bY_r\in [P^r]^M$ in $I_n$, we have the following decomposition introduced in \cite{Monk2014JSC}
\beq\label{d3}
\hat\bY_r=(-1)^r\bY^n_r\widetilde L_r+(t-t_n)\widetilde{\bY}_r,\ \ \widetilde{\bY}_r\in [P^{r-1}]^M.
\eeq
Notice that $\cP_{r-1}\hat\bY_r=\cP_{r-1}[(t-t_n)\wbY_r]$, substituting this decomposition to \eqref{g1}, we have
\ben
(-1)^r\bY^n_r\widetilde L_r'+\widetilde{\bY}_r+(t-t_n)\widetilde{\bY}_r'=\D\cP_{r-1}[(t-t_n)\widetilde{\bY}_r]\ \ \mbox{in }I_n.
\een
Multiply the equation by $\widetilde{\bY}_r\in [P^{r-1}]^M$ and integrate over $I_n$, we have by \eqref{d2} that
\ben
\frac 12\|\wbY_r\|_{L^2(I_n)}^2+\frac 12\tau_n\|\wbY_r(t_{n+1})\|_{\R^M}^2\le C\tau_n^{-1/2}r\|\bY_r^n\|_{\R^M}\|\wbY_r\|_{L^2(I_n)},
\een
where we have used the fact that by \eqref{p2}
\beq
\int_{I_n}(\D\cP_{r-1}[(t-t_n)\wbY_r],\wbY_r)dt=\int_{I_n}(t-t_n)(\D\wbY_r,\wbY_r)dt\le 0.\label{g9}
\eeq
This yields $\|\wbY_r\|_{L^2(I_n)}\le C\tau_n^{-1/2}r\|\bY^n_r\|_{\R^M}$ and thus by using \eqref{d2} 
\beq\label{g6}
\|\hat\bY_r\|_{L^2(I_n)}\le C\tau_n^{1/2}r\|\bY^n_r\|_{\R^M}.
\eeq
On the other hand, it follows from \eqref{dp} and \eqref{g1} that
\beq\label{g2}
(\bY_r-\hat\bY_r)'=\D\cP_{r-1}(\bY_r-\hat\bY_r)+\cP_{r-1}\bR\ \ \mbox{in }I_n,\ \ (\bY_r-\hat\bY_r)(t_n)=0.
\eeq
Then $\bY_r-\hat\bY_r=(t-t_n)\bW_r$ for some $\bW_r\in [P^{r-1}]^M$. By substituting this relation into the equation \eqref{g2} we have
\ben
\bW_r+(t-t_n)\bW_r'=\D\cP_{r-1}[(t-t_n)\bW_r]+\cP_{r-1}\bR\ \ \mbox{in }I_n.
\een
By multiplying the equation by $\bW_r$ and integrating over $I_n$, we obtain by a similar bound as in \eqref{g9} that
\ben
\frac 12\|\bW_r\|_{L^2(I_n)}^2+\frac 12\tau_n\|\bW_r(t_{n+1})\|_{\R^M}^2\le\|\bR\|_{L^2(I_n)}\|\bW_r\|_{L^2(I_n)}.
\een
This yields $\|\bW_r\|_{L^2(I_n)}\le 2\|\bR\|_{L^2(I_n)}$ and thus 
\beq\label{g7}
\|\bY_r-\hat\bY_r\|_{L^2(I_n)}\le 2\tau_n\|\bR\|_{L^2(I_n)}.
\eeq
Now by multiplying \eqref{g2} by $\bY_r-\hat\bY_r$ and integrating over $I_n$ we obtain by \eqref{p2} and \eqref{g7} that 
\ben
\frac 12\|(\bY_r-\hat\bY_r)(t_{n+1})\|_{\R^M}^2\le\int_{I_n}(\cP_{r-1}\bR,\bY_r-\hat\bY_r)dt\le 2\tau_n\|\bR\|_{L^2(I_n)}^2,
\een
which implies by the triangle inequality and \eqref{g5} that
\ben
\|\bY_r(t_{n+1})\|_{\R^M}\le\|\bY_r^n\|_{\R^M}+2\tau_n^{1/2}\|\bR\|_{L^2(I_n)}.
\een
This yields \eqref{g3}. Next by using the triangle inequality, \eqref{g6}, and \eqref{g7}, we have
\beq\label{g8}
\|\bY_r\|_{L^2(I_n)}\le C\tau^{1/2}_nr\|\bY^n_r\|_{\R^M}+2\tau_n\|\bR\|_{L^2(I_n)},
\eeq
which implies by the $hp$ inverse estimate that
\ben
\max_{t_n\le t\le t_{n+1}}\|\bY_r\|_{\R^M}\le C\tau_n^{-1/2} r\|\bY_r\|_{L^2(I_n)}\le Cr^2\|\bY^n_r\|_{\R^M}+C\tau_n^{1/2}r\|\bR\|_{L^2(I_n)}.
\een
This shows \eqref{g4} and completes the proof of the lemma. 
\end{proof}


To derive an $hp$ a priori error estimate for the continuous time Galerkin method \eqref{dp}, we first recall an interpolation operator in the literature (see, e.g., \cite[Theorem 3.17]{Schwab1998hp}).

\begin{lem}\label{lem:2.3}
There exists an interpolation operator $\Pi_r:[H^1(0,T)]^M\to\bV^r_\tau$ such that for any $\bv\in [W^{1+s,\infty}(0,T)]^M$, $s\ge 1$, and $n=0,1,\cdots,N-1$,
\be
& &(\Pi_r\bv)(t_n)=\bv(t_n), \ \ (\Pi_r\bv)(t_{n+1})=\bv(t_{n+1}),\ \ (\Pi_r\bv)'=\cP_{r-1}\bv'\ \ \mbox{in }I_n,\label{Pi1}\\
& &\|\bv-\Pi_r\bv\|_{L^2(I_n)}\le C\frac{\tau^{\min(r+1,s)}}{r^s}\|\bv\|_{H^s(I_n)},\label{Pi2}\\
& &\max_{t_n\le t\le t_{n+1}}\|\bv-\Pi_r\bv\|_{\R^M}\le C\frac{\tau^{\min(r,s)+1}}{r^s}\|\bv'\|_{W^{s,\infty}(I_n)},\label{Pi3}
\ee
where the constant $C$ is independent of $\tau,r$ but may depend on $s$.
\end{lem}

\begin{proof} The interpolation operator is defined as
\ben
\Pi_r\bv=\bv(t_n)+\int_{t_n}^t(\cP_{r-1}\bv')dt\ \ \ \ \forall t\in I_n.
\een
\eqref{Pi1} follows easily from this definition. Next by using \eqref{err}, we have
\ben
\max_{t_n\le t\le t_{n+1}}\|\bv-\Pi_r\bv\|_{\R^M}\le\tau_n^{1/2}\|\bv'-\cP_{r-1}\bv'\|_{L^2(I_n)}
\le C\frac{\tau_n^{\min(r,s)+1}}{r^s}\|\bv'\|_{W^{s,\infty}(I_n)}.
\een
This shows \eqref{Pi3}.

The estimate \eqref{Pi2} is proved for $s\ge 2$ in \cite{Schwab1998hp}. Here we use the duality argument to show \eqref{Pi2} also from $s\ge 1$. Let $\bw\in H^1_0(I_n)$ be the solution of the problem
\ben
-\bw''=\bv-\Pi_r\bv\ \ \mbox{in }I_n.
\een
It is easy to see that $\|\bw\|_{H^2(I_n)}\le C\|\bv-\Pi_r\bv\|_{L^2(I_n)}$. Since $(\bv-\Pi_r\bv)(t_n)=\mathbf{0}, (\bv-\Pi_r\bv)(t_{n+1})=\mathbf{0}$, we multiply the equation by $\bv-\Pi_r\bv$, integrate over $I_n$, and use \eqref{Pi1} to obtain
\ben
\|\bv-\Pi_r\bv\|_{L^2(I_n)}^2=\int_{I_n}(\bw',\bv'-(\Pi_r\bv)')dt&=&\int_{I_n}(\bw'-\cP_{r-1}\bw',\bv'-\cP_{r-1}\bv')dt\\
&\le&C\frac{\tau^{\min(r+1,s)}}{r^s}\|\bw\|_{H^2(I_n)}\|\bv\|_{H^s(I_n)}.
\een
This completes the proof by using $\|\bw\|_{H^2(I_n)}\le C\|\bv-\Pi_r\bv\|_{L^2(I_n)}$. 
\end{proof}

The following theorem on the $hp$ error estimate is the main result of this section.

\begin{thm}\label{thm:2.1}
Let $s\ge 1$. Assume that $\bR\in [H^s(0,T)]^M$, $\bY\in [W^{1+s,\infty}(0,T)]^M$ and $\bY_r\in\bV^r_\tau$ is the solution of the problem \eqref{dp}, we have
\ben
& &\max_{1\le n\le N}\|(\bY-\bY_r)(t_n)\|_{\R^M}\le CT^{1/2}\frac{\tau^{\min(r+1,s)}}{r^s}\|\D\bY\|_{H^s(0,T)},\\
& &\max_{0\le t\le T}\|\bY-\bY_r\|_{\R^M}\le C(1+T^{1/2})\frac{\tau^{\min(r+1,s)}}{r^{s-2}}(T^{1/2}\|\bY\|_{W^{s+1,\infty}(0,T)}+\|\bR\|_{H^s(0,T)}),
\een
where the constant $C$ is independent of $\tau,r,\D$ but may depend on $s$.
\end{thm}

\begin{proof}
Let $\Pi_r\bY\in\bV^r_\tau$ be the interpolation of $\bY$ defined in Lemma \ref{lem:2.3}. Since $(\Pi_r\bY)'=\cP_{r-1}\bY'$ in $I_n$, we have
\ben
(\Pi_r\bY)'=\cP_{r-1}(\D\bY+\bR)=\D\cP_{r-1}(\bY-\Pi_r\bY)+\D\cP_{r-1}(\Pi_r\bY)+\cP_{r-1}\bR\ \ \mbox{in }I_n.
\een
Thus by \eqref{dp} we have
\beq
\bY_r'-(\Pi_r\bY)'=\D\cP_{r-1}(\bY_r-\Pi_r\bY)-\D\cP_{r-1}(\bY-\Pi_r\bY).\label{err_equ}
\eeq
As $(\bY_r-\Pi_r\bY)(0)=\mathbf{0}$, we use \eqref{g3} and \eqref{Pi2} to obtain
\ben
\max_{1\le n\le N}\|(\bY_r-\Pi_r\bY)(t_n)\|_{\R^M}&\le&CT^{1/2}\|\D\cP_{r-1}(\bY-\Pi_r\bY)\|_{L^2(0,T)}\nn\\
&\le&CT^{1/2}\frac{\tau^{\min(r+1,s)}}{r^{s}}\|\D\bY\|_{H^s(0,T)}.
\een
This shows the first estimate as $\bY(t_n)=\Pi_r\bY(t_n)$. The second estimate can be proved similarly by using \eqref{g4}, \eqref{Pi3}, and $\bY'=\D\bY+\bR$. 
\end{proof}

We remark that the first estimate in Theorem \ref{thm:2.1} is optimal in $\tau$ and $r$ and the second estimate in the theorem is optimal in $\tau$ but suboptimal in $r$ which is due to the stability estimate \eqref{g4} in Lemma \ref{lem:2.2}. In section 5 we will show that the stability in the $L^2$ norm can be improved to remove the dependence on $r$ in \eqref{g8} when $\D$ is symmetric or skew-symmetric by using the explicit formulas of $\bY_r(t)$ in section 3. We remark that 
many spatial discretization matrices of the wave-like equations satisfy the property that $\D$ is skew-symmetric, such as the energy conserving mixed finite element methods for solving the Hodge wave equation in Wu and Bai \cite{Wu2021SINUM} and the unfitted finite element method of the acoustic wave equation in Chen et. al. \cite{Chen2021preprint}. 


The classical order of Runge-Kutta methods is the convergence order at the nodes $t=t_n$. For the continuous time Galerkin method, it is proved to be ${2r}$ when $r\ge 2$ in Hulme \cite{Hulme} for nonlinear ODEs and in Aziz and Monk \cite{Aziz} for parabolic equations. The following theorem shows the $hp$ superconvergence of the continuous time Galerkin method at the nodes by using the idea of quasi-projection in \cite[\S 4]{Aziz}.

\begin{thm}\label{super_conver}
Let $s\ge 1$. Assume that $\D^r\bY\in [H^{s}(0,T)]^M$ and $\bY_r\in\bV^r_\tau$ is the solution of the problem \eqref{dp}, we have
\ben
\max_{1\le n\le N}\|(\bY-\bY_r)(t_n)\|_{\R^M}\le CT^{1/2}\frac{\tau^{\min(2r,s+r-1)}}{r^{s}}\|\D^{r}\bY\|_{H^{s}(0,T)},
\een
where the constant $C$ is independent of $\tau,r$ but may depend on $s$.
\end{thm}
\begin{proof}
If $r=1$, the theorem follows from the first estimate of Theorem \ref{thm:2.1}.
Now we assume $r\ge 2$. Let $\Pi_r\bY\in\bV^r_\tau$ be the interpolation of $\bY$ defined in Lemma \ref{lem:2.3}. Denote $\bm{\omega}_0=\bY-\Pi_r\bY$. For $1\leq i\leq r-1$, we define correction functions $\bm{\omega}_i$ such that
\be
\bm{\omega}_i(t_n)=\mathbf{0},\,  \bm{\omega}_i'=\D\cP_{r-1}\bm{\omega}_{i-1} \ \ \text{ in } I_n,\ \  n=0,1,\cdots,N-1. \label{correct_fun}
\ee
We claim that $\bm{\omega}_i(t_{n+1})=\mathbf{0}$ so that $\bm{\omega}_i\in\bV^r_\tau$. In fact,
by \eqref{Pi1}, we have $(\bm{\omega}_0',\bv)_{I_n}=0$ for any $\bv\in [P^{r-1}]^M$, where $(\cdot,\cdot)_{I_n}$ is the
inner product of $[L^2(I_n)]^M$. Since $\bm{\omega}_0(t_n)=\bm{\omega}_0(t_{n+1})=\bm{0}$, we obtain by integration by parts that $(\bm{\omega}_1',\bv')_{I_n}=(\mathbb{D}\bm{\omega}_0,\bv')_{I_n}=0$.
Therefore, $(\bm{\omega}_1',\bv)_{I_n}=0$ for any $\bv\in [P^{r-2}]^M$ and consequently, $\bm{\omega}_1(t_{n+1})=\int_{I_n}\bm{\omega}_1'dt=\mathbf{0}$. By mathematical induction, we know easily by the same argument that
$(\bm{\omega}_i',\bv)_{I_n}=0$ for any $\bv\in [P^{r-i-1}]^M$ and $\bm{\omega}_i(t_{n+1})=\mathbf{0}$. This shows the claim.

Let $\displaystyle\bm{\omega}=\sum_{i=1}^{r-1}\bm{\omega}_{i}\in \bV^r_\tau$. By \eqref{err_equ} and \eqref{correct_fun}, we have
\ben
\bY_r'-(\Pi_r\bY)'+\bm{\omega}'=\D\cP_{r-1}(\bY_r-\Pi_r\bY+\bm{\omega})-\D\cP_{r-1}(\bm{\omega}_{r-1}).
\een
As $(\bY_r-\Pi_r\bY+\bm{\omega})(0)=\mathbf{0}$, we use \eqref{g3} to obtain
\ben
\|(\bY_r-\Pi_r\bY+\bm{\omega})(t_n)\|_{\R^M}&\le&CT^{1/2}\|\D\cP_{r-1}(\bm{\omega}_{r-1})\|_{L^2(0,T)}
\le CT^{1/2}\|\D\bm{\omega}_{r-1}\|_{L^2(0,T)}.
\een
Now it follows from \eqref{correct_fun} that
\ben
\|\D\bm{\omega}_i\|_{L^2(I_n)}\le \tau\|\D^2\bm{\omega}_{i-1}\|_{L^2(I_n)},\quad 1\leq i\leq r-1.
\een
By using \eqref{Pi2} we have then
\ben
\|\D\bm{\omega}_{r-1}\|_{L^2(I_n)}\le \tau^{r-1}\|\D^r\bm{\omega}_0\|_{L^2(I_n)}\le C\frac{\tau^{\min(2r,s+r-1)}}{r^s}\|\D^{r}\bY\|_{H^s(I_n)}.
\een
This completes the proof since by \eqref{Pi1} and \eqref{correct_fun}, $(\bY_r-\Pi_r\bY+\bm{\omega})(t_n)=(\bY_r-\bY)(t_n)$, $1\le n\le N$.   
\end{proof}

The correction function $\bm{\omega}=\sum^{r-1}_{i=1}\bm{\omega}_i$ is introduced in \cite{Aziz} which is related to the idea of quasi-projection in Douglas Jr. et al \cite{Douglas}. Our new observation is that $\bm{\omega}_i=\bf{0}$ at the nodes for $1\le i\le r-1$, which simplifies the proof.


To conclude this section, we recall some facts about Pad\'e approximation to the exponential function which can be found in Saff and Varga \cite{saff1975Num} and the references therein. For any integers $m,n\ge 0$, the $[m/n]$ Pad\'e approximation to $e^z$ is defined as the polynomials $P_m(z)\in P^m$, $Q_n(z)\in P^n$, $Q_n(0)=1$, for which
\ben
e^z-\frac{P_m(z)}{Q_n(z)}=O(|z|^{m+n+1})\ \ \ \ \mbox{as }|z|\to 0.
\een
It is known that
\beq\label{pade}
P_m(z)=\sum^m_{j=0}\frac{(m+n-j)!m!z^j}{(m+n)!j!(m-j)!},\ \ Q_n(z)=\sum^n_{j=0}\frac{(m+n-j)!n!(-z)^j}{(m+n)!j!(n-j)!}.
\eeq
Obviously, $Q_n(z)=P_n(-z)$. When $m=n$, $P_m(z), Q_m(z)$ are called diagonal Pad\'e numerator and denominator of type $[m/m]$ for $e^z$. The following lemma follows easily from \eqref{pade}

\begin{lem}\label{lem:2.4}
The diagonal Pad\'e numerator of type $[m/m]$ for $e^z$ satisfies $P_0(z)=1, P_1(z)=1+\frac 12z$, $P_2(z)=1+\frac 12 z+\frac 1{12}{z^2}$, and
\ben
P_m(z)=P_{m-1}(z)+\frac{z^2}{4(2m-1)(2m-3)}P_{m-2}(z),\ \ \ \ m\ge 2.
\een
\end{lem}

The following lemma is proved in Hairer and Wanner \cite[Theorem 4.12]{Hairer1996}, \cite[Theorem 2.4]{saff1975Num}. It is essential in proving the A-stability of numerical methods
for ODEs based on the Pad\'e approximation of the exponential function.

\begin{lem}\label{lem:2.5}
All zeros of the diagonal Pad\'e numerator of type $[m/m]$, $m\ge 1$, for $e^z$ are simple and lie in the half-plane $\{z\in\C:\Re (z)\le-2\}$.
\end{lem}

For $m\ge 1$, denote $\zeta_1,\cdots,\zeta_{m}\in\C$ the zeros of $P_m(z)$, the diagonal Pad\'e numerator of type $[m/m]$ for $e^z$, then by \eqref{pade}
\ben
P_m(z)=\frac{m!}{(2m)!}(z-\zeta_1)\cdots(z-\zeta_m),\ \ Q_m(z)=(-1)^m\frac{m!}{(2m)!}(z+\zeta_1)\cdots(z+\zeta_m).
\een
Recall that any polynomial $F\in P^{m-1}$ can be expanded as the Lagrange interpolation function at $m$ distinct zeros of $P_m(-z)$. This yields the following partial fraction formula (see, e.g., Szeg\"o \cite[Theorem 3.3.5]{Szego})
\beq\label{pade1}
\frac{F(z)}{P_m(-z)}=\sum^m_{j=1}-\frac{F(-\zeta_j)}{P_m'(\zeta_j)}\,\frac{1}{z+\zeta_j}.
\eeq
Since $P_m(z)-(-1)^mP_m(-z)\in P^{m-1}$, we obtain the partial fraction formula for the $[m/m]$ Pad\'e approximation of $e^z$ (see Gallopoulos and Saad \cite{Saad1989})
\beq\label{pade2}
R_{m,m}(z)=\frac{P_m(z)}{P_m(-z)}=(-1)^m+\sum^m_{j=1}-\frac{P_m(-\zeta_j)}{P_m'(\zeta_j)}\,\frac{1}{z+\zeta_j}.
\eeq
Recall that if $p(z)=\sum^m_{i=0}a_i z^i$ is a polynomial of degree $m$, then $p(\X):=\sum^m_{i=0}a_i\X^i$ for any matrix $\X\in\R^{d\times d}$, $d\ge 1$. Obviously, if $p(z)=p_1(z)+p_2(z)$ or $q(z)=p_1(z)\cdot p_2(z)$, where $p_1,p_2$ are polynomials, then $p(\X)=p_1(\X)+p_2(\X)$, $q(\X)=p_1(\X)\cdot p_2(\X)$. It follows now from \eqref{pade1}-\eqref{pade2} that for any $F\in P^{m-1}$ and any matrix $\X\in\R^{d\times d}$ such that $P_m(-\X)$ is invertible,
\be
& &\frac{F(\X)}{P_m(-\X)}=\sum^m_{j=1}-\frac{F(-\zeta_j)}{P_m'(\zeta_j)}(\zeta_j\I+\X)^{-1},\label{pade3}\\
& &\frac{P_m(\X)}{P_m(-\X)}=(-1)^m\I+\sum^m_{j=1}-\frac{P_m(-\zeta_j)}{P_m'(\zeta_j)}\,(\zeta_j\I+\X)^{-1}.\label{pade4}
\ee
The identity \eqref{pade4} is the basis of the method \eqref{a1} in the introduction.

\section{Parallel implementation}\label{sec_implementation}
In this section, we propose parallel algorithms to implement the problem \eqref{dp} based on finding analytic formulas of the determinant and all factors of the stiffness matrix of the continuous time Galerkin method at each time step. To form the stiffness matrix, we use the Legendre polynomials $\{\widetilde L_j\}^r_{j=0}$ in $I_n$ as the basis functions.
We assume
\begin{align*}
\bY_r(t)=\sum_{j=0}^{r}\bfa_j\widetilde{L}_j(t),\quad \mathcal{P}_{r-1}\bR=\sum_{j=0}^{r-1}\bR_j\widetilde{L}_j(t),
\end{align*}
where $\bfa_j,\bR_j\in\mathbb{R}^M$. Since $\cP_{r-1}\widetilde L_r=0$ in $I_n$, from \eqref{dp}, we have
\begin{align*}
\sum_{j=0}^{r}\bfa_j\widetilde{L}_j'(t)=\D\sum_{j=0}^{r-1}\bfa_j\widetilde{L}_j(t)+\sum_{j=0}^{r-1}\bR_j\widetilde{L}_j(t).
\end{align*}
For any $k\ge 1$, multiplying the equation by $(t-t_n)(t_{n+1}-t)\widetilde{L}'_k(t)$ and integrating over $I_n$, we obtain
\begin{align}
\bfa_k\frac{\tau_n}{2}\frac{k(k+1)}{k+\frac 12}=&\D\sum_{j=0}^{r-1}\int_{t_n}^{t_{n+1}}\bfa_j\widetilde{L}_j(t)(t-t_n)(t_{n+1}-t)\widetilde{L}_k'(t)\, dt\nonumber \\
&+\sum_{j=0}^{r-1}\int_{t_n}^{t_{n+1}}\bR_j\widetilde{L}_j(t)(t-t_n)(t_{n+1}-t)\widetilde{L}_k'(t)\, dt,\label{y1}
\end{align}
where we have used the fact that
\begin{align*}
\int_{t_n}^{t_{n+1}}\widetilde{L}_j'\widetilde{L}_k'(t-t_n)(t_{n+1}-t)\, dt=\frac{\tau_n}2\int^1_{-1}L_j'L_k'(1-t^2)dt=\frac{\tau_n}{2}\frac{k(k+1)}{k+\frac 12}\delta_{j,k}.
\end{align*}
Here $\delta_{j,k}$ is the Kronecker delta function. By the recursion relation $(2k+1)\widetilde{L}_k(t)=\frac{\tau_n}{2}(\widetilde{L}'_{k+1}(t)-\widetilde{L}'_{k-1}(t))$,
\begin{align*}
&\int_{t_n}^{t_{n+1}}\widetilde{L}_j(t)(t-t_n)(t_{n+1}-t)\widetilde{L}'_k(t)\, dt\\
=&\,\frac{1}{2j+1}\int_{t_n}^{t_{n+1}}\frac{\tau_n}{2}(\widetilde{L}'_{j+1}(t)-\widetilde{L}'_{j-1}(t))\widetilde{L}'_k(t)(t-t_n)(t_{n+1}-t)\, dt\nonumber\\
=&\,\frac{\tau_n^2}{4}\frac{k(k+1)}{k+\frac 12}\left(\frac{1}{2k-1}\delta_{j+1,k}-\frac{1}{2k+3}\delta_{j-1,k}\right).
\end{align*}
Substituting the identity into \eqref{y1}, we have
\begin{align}
&\bfa_k=\frac{\tau_n}{2}\D\left(\frac{\bfa_{k-1}}{2k-1}-\frac{\bfa_{k+1}}{2k+3}\right)+\frac{\tau_n}{2}\left(\frac{\bR_{k-1}}{2k-1}-\frac{\bR_{k+1}}{2k+3}\right),\label{coef_1} \quad 1\leq k\leq r-2,\\
&\bfa_{k}=\frac{\tau_n}{2}\D\frac{\bfa_{k-1}}{2k-1}+\frac{\tau_n}{2}\frac{\bR_{k-1}}{2k-1}, \quad k=r-1,r.\label{coef_2}
\end{align}
By the condition $\bY_r(t_{n})=\bY_r^n$, we also have
\begin{align}\label{coef_4}
\sum_{j=0}^{r}(-1)^j\bfa_j=\bY_r^n.
\end{align}
\eqref{coef_1}-\eqref{coef_4} can be written as a system of linear equations
\begin{align}\label{y2}
\A\mathbf{X}=\bB,
\end{align}
where $\mathbf{X}=(\bfa_0^T,\bfa_1^T,\cdots,\bfa_{r}^T)^T$, $\bB=(\bfb_0^T,\bfb_1^T,\cdots,\bfb_{r}^T)^T$ with 
\ben
\bfb_{k-1}=\left\{\begin{array}{ll}
-\frac{\tau_n}2\left(\frac{\bR_{k-1}}{2k-1}-\frac{\bR_{k+1}}{2k+3}\right) & \mbox{if } 1\le k\le r-2,\\
-\frac{\tau_n}2\frac{\bR_{k-1}}{2k-1} & \mbox{if } k=r-1,r, \\
\bY^n_r & \mbox{if }k=r+1,
\end{array}\right.
\een
and
\begin{align*}
\A=\left(\begin{array}{ccccccc}
\frac{\tau_n}{2}\D&-\I&-\frac{\tau_n}{2}\D\frac{1}{5}&\cdots&\cdots&\cdots&0\\
0&\frac{\tau_n}{2}\D\frac{1}{3}&-\I&-\frac{\tau_n}{2}\D\frac{1}{7}&\cdots&\cdots&0\\
\vdots&\vdots&\ddots&\ddots&\ddots&\vdots&\vdots\\
0&\cdots&\cdots&\frac{\tau_n}{2}\D\frac{1}{2r-5}&-\I&-\frac{\tau_n}{2}\D\frac{1}{2r-1}&0\\
0&\cdots&\cdots&\cdots&\frac{\tau_n}{2}\D\frac{1}{2r-3}&-\I&0\\
0&\cdots&\cdots&\cdots&\cdots&\frac{\tau_n}{2}\D\frac{1}{2r-1}&-\I\\
\I&-\I&\I&-\I&\cdots&\cdots&(-1)^{r+2}\I
\end{array}\right).
\end{align*}
Here $\I\in\mathbb{R}^{M\times M}$ is the identity matrix. By Lemma \ref{lem:2.2}, \eqref{y2} has a unique solution.
Since $\A\in \mathbb{R}^{M(r+1)\times M(r+1)}$, it is expensive to solve \eqref{y2} directly when $M\gg 1$.
Here we propose efficient and parallel algorithms to solve \eqref{y2}.

Notice that $\A$ is a $(r+1)\times (r+1)$ block matrix. For $\lambda\in\R$, we define $\E_{r+1}(\lam)\in\R^{(r+1)\times(r+1)}$ by
\begin{align}\label{defEr}
\E_{r+1}(\lambda):=\left(\begin{array}{ccccccc}
a_1&-1&b_1&\cdots&\cdots&\cdots&0\\
0&a_2&-1&b_2&\cdots&\cdots&0\\
\vdots&\vdots&\ddots&\ddots&\ddots&\vdots&\vdots\\
0&\cdots&\cdots&a_{r-2}&-1&b_{r-2}&0\\
0&\cdots&\cdots&\cdots&a_{r-1}&-1&0\\
0&\cdots&\cdots&\cdots&\cdots&a_r&-1\\
c_1&c_2&\cdots&\cdots&\cdots&c_r&c_{r+1}\\
\end{array}\right)
\end{align}
with
\beq\label{y3}
a_k=\frac \lam2\frac 1{2k-1},\ \  b_k=-\frac \lam 2\frac{1}{2k+3},\ \ c_k=(-1)^{k+1},\ \ k=1,\cdots,r+1.
\eeq
Then $\A=\E_{r+1}(\tau_n\D)$ by replacing each element $e_{ij}(\lam)$ of $\E_{r+1}(\lam)$ by the $M\times M$ matrix $e_{ij}(\tau_n\D)\I$, $i,j=1,\cdots,r+1$.

We are going to solve \eqref{y2} by extending the Cramer rule for the block matrices. We first introduce some notation. For any matrix $\X=(X_{ij})^d_{i,j=1}$, $d\ge 1$, we denote $\X_{i,j}$ the matrix obtained by removing the $i$-th row and $j$-th column, $i,j=1,\dots,d$. We also denote $\X_{(i_1,\cdots,i_k),(j_1,\cdots,j_l)}$ the matrix obtained by removing the $i_1,\cdots,i_k$-th rows and the $j_1,\cdots,j_l$-th columns, where $1\le i_1<\cdots<i_k\le d, 1\le j_1<\cdots<j_l\le d$. The following property about the adjugate matrix is well known
\beq\label{y4}
(\det\X)\delta_{i,j}=\sum^d_{k=1}(-1)^{i+k}(\det\X_{k,i})X_{kj}.
\eeq

Denote by $\mathbb{H}=\E_{r+1}(\lam)_{(r-1,r,r+1),(r-1,r,r+1)}\in\R^{(r-2)\times(r-2)}$. Then the matrix $\E_{r+1}(\lam)$ can be partitioned as
\begin{align}\label{c1}
\E_{r+1}(\lam)=\left(\begin{array}{ccc|ccc}
  & & & & & \\
 & \mathbb{H}& &b_{r-3}& 0&0 \\
 & & &-1&b_{r-2}&0\\\cline{1-3}
\multicolumn{3}{c}{ }&a_{r-1}&-1&0\\
\multicolumn{3}{c}{ }&0&a_{r}&-1\\
c_1&\cdots &\multicolumn{1}{c}{c_{r-2}}&c_{r-1}&c_{r}&c_{r+1}
\end{array}\right).
\end{align}
Similarly, we have
\begin{align}\label{c2}
\E_{r}(\lam)=\left(\begin{array}{ccc|cc}
  & & & &  \\
 & \mathbb{H}& &b_{r-3}&0 \\
 & & &-1&0\\\cline{1-3}
\multicolumn{3}{c}{ }&a_{r-1}&-1\\
c_1&\cdots &\multicolumn{1}{c}{c_{r-2}}&c_{r-1}&c_{r}
\end{array}\right),\ \
\E_{r-1}(\lam)=\left(\begin{array}{ccc|c}
  & & &   \\
 & \mathbb{H}& &0\\
 & & &-1\\\cline{1-3}
c_1&\cdots &\multicolumn{1}{c}{c_{r-2}}&c_{r-1}
\end{array}\right).
\end{align}
The following simple identities will play an important role in our analysis
\beq\label{c3}
\E_{r+1}(\lam)_{r,(r,r+1)}=\E_r(\lam)_{*,r},
\eeq
where for any $\X\in\R^{d\times d}$, we denote $\X_{*,j}\in\R^{d\times (d-1)}$ the matrix by removing the $j-$th column of $\X$. Similarly, we denote $\X_{i,*}\in\R^{(d-1)\times d}$ the matrix by removing the $i$-th row of $\X$.

The following elementary lemma is useful in our analysis.

\begin{lem}\label{lem:3.0}
For any $r\ge 3$ and $1\le j\le r-2$, we have
\begin{align}
\det[\E_{r+1}(\lam)_{(r-2,r,r+1),(j,r,r+1)}]&=-a_{r-1}\det[\E_{r-1}(\lam)_{r-1,j}]\nn\\
&=a_{r-1}c_{r-1}\det[\E_{r-1}(\lam)_{r-2,j}],\label{c4}\\
\det[\E_{r+1}(\lam)_{(r-1,r,r+1),(j,r,r+1)}]&=-\det[\E_r(\lam)_{r,j}]\nn\\
&=c_r\det[\E_r(\lam)_{r-1,j}].\label{c5}
\end{align}
\end{lem}

\begin{proof} We only prove \eqref{c4}. The identity \eqref{c5} can be proved similarly. By \eqref{c3},
\ben
\det[\E_{r+1}(\lam)_{(r-2,r,r+1),(j,r,r+1)}]&=&\det[\E_r(\lam)_{(r-2,r),(j,r)}]\\
&=&\det\left(\begin{array}{ccc|c}
  & & &   \\
 & \mathbb{H}_{r-2,j}& &\\
 & & &b_{r-3}\\\cline{1-3}
&&\multicolumn{1}{c}{}&a_{r-1}
\end{array}\right)\\
&=&a_{r-1}\det[\mathbb{H}_{r-2,j}].
\een
On the other hand, it is easy to see from \eqref{c3} that
\ben
\det[\E_{r-1}(\lam)_{r-1,j}]=-\det\mathbb{H}_{r-2,j},\ \ \det[\E_{r-1}(\lam)_{r-2,j}]=c_{r-1}\det[\mathbb{H}_{r-2,j}].
\een
This completes the proof.  
\end{proof}

To proceed, we note that
\beq\label{y6}
\E_2(\lam)=\left(\begin{array}{cc}
a_1 & -1 \\
1 & -1
\end{array}\right),\ \ \ \ \E_3(\lam)=\left(\begin{array}{ccc}
a_1  & -1 & 0 \\
0 & a_2 & -1 \\
1 & -1 & 1
\end{array}\right).
\eeq

\begin{lem}\label{lem:3.1}
Let $\varphi_r(\lam)=\det\E_{r+1}(\lam)$ be the determinant of $\E_{r+1}(\lam)$. Then $\varphi_1(\lambda)=1-a_1$, $\varphi_2(\lambda)=1-a_1+a_1a_2$, and
\begin{align}
\varphi_r(\lambda)=\varphi_{r-1}(\lambda)+a_ra_{r-1}\varphi_{r-2}(\lambda),\quad r\geq 3.\label{detE_1}
\end{align}
Moreover, $\varphi_r(\lambda)=P_r(-\lambda)$, where $P_r(\lambda)$ is the numerator of $[r/r]$ Pad\'{e} approximation of $e^\lambda$.
\end{lem}

\begin{proof}
The determinants of $\E_{r+1}(\lambda)$ for $r=1,2$ follow easily from \eqref{y6}. Since by \eqref{y3}, $a_r=-b_{r-2}$, we obtain by adding the $r$-th row to the $(r-2)$-th row of $\E_{r+1}(\lam)$ in \eqref{c1} and then expanding the determinant by the $r$-th row that
\ben
\varphi_r(\lam)&=&\det\left(\begin{array}{ccc|ccc}
  & & & & & \\
 & \mathbb{H}&& b_{r-3}&0& 0 \\
 & & &-1&0&-1\\\cline{1-3}
\multicolumn{3}{c}{ }&a_{r-1}&-1&0\\
\multicolumn{3}{c}{ }&0&a_{r}&-1\\
c_1&\cdots &\multicolumn{1}{c}{c_{r-2}}&c_{r-1}&c_{r}&c_{r+1}
\end{array}\right)\\
&=&\det\E_r(\lam)+a_r\det\left(\begin{array}{ccc|cc}
  & & & &  \\
 & \mathbb{H}& &b_{r-3}&0 \\
 & & &-1&-1\\\cline{1-3}
\multicolumn{3}{c}{ }&a_{r-1}&0\\
c_1&\cdots &\multicolumn{1}{c}{c_{r-2}}&c_{r-1}&c_{r}
\end{array}\right)\\
&=&\det\E_{r}(\lam)+a_ra_{r-1}\det\E_{r-1}(\lam),
\een
where in the last equality we have expanded the determinant by the $(r-1)$-th row. Since $a_1=\lam/2$, $a_2=\lam/6$,  we use Lemma \ref{lem:2.4} to conclude $\varphi_r(\lam)=P_r(-\lam)$, where $P_r(\lambda)$ is the numerator of $[r/r]$ Pad\'{e} approximation of $e^\lambda$.   
\end{proof}

\begin{thm}\label{thm:3.1}
The matrix $\varphi_{r}(\tau_n\D)$ is invertible. The solution of \eqref{y2} $\mathbf{X}=(\bfa_0^T,\cdots,\bfa_r^T)^T$ satisfies that for $i=1,\cdots,r+1$,
\ben
\bfa_{i-1}=(-1)^r\delta_{i,r+1}\bfb_{r}+\sum^r_{j=1}(\zeta_j\mathbb{I}+\tau_n\D)^{-1}\left(\sum^{r+1}_{k=1}(-1)^{i+k+1}\frac{\phi_{ki}(-\zeta_j)}{P'_r(\zeta_j)}\,\bfb_{k-1}\right),
\een
where $\phi_{ki}(\lam)=\det[\E_{r+1}(\lam)_{k,i}]$, $k,i=1,\cdots,r+1$, are the minors of $\E_{r+1}(\lam)$.
\end{thm}

\begin{proof}
By Lemma \ref{lem:3.1}, $\varphi_r(\lam)=P_r(-\lam)=(-1)^r\frac{r!}{(2r)!}(\lam+\zeta_1)\cdots(\lam+\zeta_r)$, where $\zeta_1,\cdots,\zeta_r\in\C$ are zeros of the diagonal Pad\'e numerator of type $[r/r]$ for $e^\lam$. By Lemma \ref{lem:2.5}, $\Re(\zeta_k)\le -2$, $k=1,\cdots,r$. On the other hand, \eqref{p2} implies that the eigenvalues of $\D$ lie in the left half-plane. Thus the eigenvalues of $\zeta_k\I+\tau_n\D$, $1\le k\le r$, lie in the half-plane $\{z\in\C:\Re (z)\le -2\}$. This shows $\varphi_r(\tau_n\D)$ is invertible.

Now by \eqref{y4} we have
\ben
[\det\E_{r+1}(\lam)]\delta_{i,j}=\sum^{r+1}_{k=1}(-1)^{i+k}[\det\E_{r+1}(\lam)_{k,i}]e_{kj}(\lam),
\een
where $e_{kj}(\lam)$ is the $(k,j)$ element of $\E_{r+1}(\lam)$. By replacing $\lam$ by $\tau_n\D$ in above equality, we have
\beq\label{y5}
\varphi_r(\tau_n\D)\delta_{i,j}=\sum^{r+1}_{k=1}(-1)^{i+k}\phi_{ki}(\tau_n\D)e_{kj}(\tau_n\D).
\eeq
From \eqref{y2} we have
\ben
\sum^{r+1}_{j=1}e_{ij}(\tau_n\D)\bfa_{j-1}=\bfb_{i-1},\ \ \ i=1,\cdots,r+1.
\een
Thus multiplying \eqref{y5} by $\bfa_{j-1}$ and summing over $j$ from $1$ to $r+1$, we obtain
\be
\varphi_r(\tau_n\D)\bfa_{i-1}&=&\sum^{r+1}_{k=1}(-1)^{i+k}\phi_{ki}(\tau_n\D)\cdot\sum^{r+1}_{j=1}e_{kj}(\tau_n\D)\bfa_{j-1}\nn\\
&=&\sum^{r+1}_{k=1}(-1)^{i+k}\phi_{ki}(\tau_n\D)\bfb_{k-1}.\label{y7}
\ee
Note that for $(k,i)\not=(r+1,r+1)$, $\phi_{ki}(\lambda)\in P^m$, $m\le r-1$, by \eqref{pade3} we have
\ben
\frac{\phi_{ki}(\tau_n\D)}{\varphi_r(\tau_n\D)}=\sum^r_{j=1}-\frac{\phi_{ki}(-\zeta_j)}{P_r'(\zeta_j)}(\zeta_j\I+\tau_n\D)^{-1}.
\een
For $(k,i)=(r+1,r+1)$, we have from \eqref{defEr} that $\det[E_{r+1}(\lambda)_{r+1,r+1}]=\prod_{j=1}^{r}a_j=\frac{r!}{(2r)!}\lambda^r$. Thus $\phi_{r+1,r+1}(\lambda)-(-1)^r\varphi_r(\lambda)\in P^{r-1}$, by using \eqref{pade3} again we obtain
\ben
\frac{\phi_{r+1,r+1}(\tau_n\D)}{\varphi_r(\tau_n\D)}=(-1)^r\I+\sum^r_{j=1}-\frac{\phi_{r+1,r+1}(-\zeta_j)}{P_r'(\zeta_j)}(\zeta_j\I+\tau_n\D)^{-1}.
\een
This completes the proof of the theorem.  
\end{proof}

We remark that \eqref{y7} can also be proved by using an abstract result in Brown \cite[Theorem 2.19 and Corollary 2.21]{Brown}
where linear algebra when matrix elements are defined over a space of commuting matrices are studied.

From this theorem we know that the discrete problem \eqref{dp} can be solved by solving $r(r+1)$ linear systems of equations of order $M\times M$ in parallel once all minors of $\E_{r+1}(\lam)$ at $\lam=-\zeta_j, j=1,\cdots,r$, are known. In the following, we will find recursive formulas to computing these minors.

Let $\G_r(\lam)=\E_{r+1}(\lam)_{r,r+1}$, $r\ge 1$. The determinants of $\G_r(\lam)$ for $r=1,2$ can be calculated by \eqref{y6}. We have the following lemma for $\det\G_r(\lam)$ for $r\ge 3$.

\begin{lem}\label{lem:3.2}
For $r\ge 3$, we have
\ben
\det\G_r(\lam)=\det\G_{r-1}(\lam)-a_{r-1}b_{r-2}\det\G_{r-2}(\lam)+c_r\prod^{r-1}_{k=1}a_k.
\een
\end{lem}
\begin{proof} By definition and the partition in \eqref{c1}, we know that
\ben
\G_r(\lam)=\left(\begin{array}{ccc|cc}
  & & & &  \\
 & \mathbb{H}& &b_{r-3}&0 \\
 & & &-1&b_{r-2}\\\cline{1-3}
\multicolumn{3}{c}{ }&a_{r-1}&-1\\
c_1&\cdots &\multicolumn{1}{c}{c_{r-2}}&c_{r-1}&c_{r}
\end{array}\right).
\een
By expanding the determinant by the $(r-1)$-th row and use \eqref{c3}, we obtain
\ben
\det\G_r(\lam)&=&\det[\E_{r+1}(\lam)_{(r-1,r),(r,r+1)}]+a_{r-1}\det\left(\begin{array}{ccc|c}
  & & &   \\
 & \mathbb{H}& & \\
  & & & b_{r-2}\\\cline{1-3}
c_1&\cdots &\multicolumn{1}{c}{c_{r-2}}&c_{r}
\end{array}\right)\\
&=&\det[\E_{r}(\lam)_{r-1,r}]+a_{r-1}(c_r\det\mathbb{H}-b_{r-2}\det[\E_{r+1}(\lam)_{(r-2,r-1.r),(r-1,r,r+1)}])\\
&=&\det\G_{r-1}(\lam)+c_r\prod^{r-1}_{k=1}a_k-a_{r-1}b_{r-2}\det[\E_{r}(\lam)_{(r-2,r-1),(r-1,r)}]\\
&=&\det\G_{r-1}(\lam)+c_r\prod^{r-1}_{k=1}a_k-a_{r-1}b_{r-2}\det[\E_{r-1}(\lam)_{r-2,r-1}]\\
&=&\det\G_{r-1}(\lam)-a_{r-1}b_{r-2}\det\G_{r-2}(\lam)+c_r\prod^{r-1}_{k=1}a_k,
\een
where we have used the fact that $\det\mathbb{H}=\Pi^{r-2}_{k=1}a_k$ and expanded the determinant by the last column in the second equality. This completes the proof.  
\end{proof}

The minors of $\G_r(\lam)$ for $r=1,2$ can be computed directly by \eqref{y6}. The following lemma gives the recursive formulas for some of the minors of $\G_r(\lam)$ which will be used to compute the minors of $\E_{r+1}(\lam)$.

\begin{lem}\label{lem:3.3} For $r\ge 3$, we have
\ben
\det[\G_r(\lam)_{i,r}]&=&(-1)^{r-i-1}\frac{a_1\cdots a_{r-1}}{a_1\cdots a_i}\det\G_i(\lam),\ \ 1\le i\le r-1,\\
\det[\G_r(\lam)_{r-1,j}]&=&\det[\E_r(\lam)_{r-1,j}]-b_{r-2}\det[\G_{r-1}(\lam)_{r-2,j}],\ \ 1\le j\le r-1,\\
\een
\end{lem}
\begin{proof} For $i=1,\cdots,r-2$, by definition, we have
\ben
\det[\G_r(\lam)_{i,r}]&=&\det\left(\begin{array}{ccccccccc}
a_1&-1&b_1& & & & & &\\
&\ddots&\ddots &\ddots &&&&\\
& &a_{i-1}  &-1 & b_{i-1} &&&&\\
& & &0&a_{i+1}& -1& b_{i+1}&& \\
& & & &\ddots &\ddots &\ddots &\ddots&\\
& & & & & 0 &a_{r-3} & -1 & b_{r-3}\\
& & & && & 0 &a_{r-2} & -1\\
& & & & && & 0 & a_{r-1}\\
c_1&\cdots&\cdots&\cdots&\cdots&\cdots& c_{r-3}&c_{r-2}&c_{r-1}\\
\end{array}\right)\\
&=&(-a_{r-1})\cdots(-a_{i+1})\det\left(\begin{array}{ccccc}
a_1&-1&b_1&&\\
&\ddots &\ddots &\ddots &\\
& & a_{i-2}&-1&b_{i-2}\\
& & & a_{i-1}&-1\\
c_1&\cdots&\cdots&c_{i-1}&c_i\\
\end{array}\right)\\
&=&(-1)^{r-i-1}(\prod^{r-1}_{k=i+1}a_k)\det\G_i(\lam).
\een
Finally, for $\G(\lam)_{i,r}$, $i=r-1$, we have by the definition and using the first identity in \eqref{c3}
\ben
\G_r(\lam)_{r-1,r}=\E_{r+1}(\lam)_{(r-1,r),(r,r+1)}=\E_r(\lam)_{r-1,r}=\G_{r-1}(\lam).
\een
This shows the first equality of the lemma. To show the second equality, for any $1\le j\le r-2$, we have by using the partition \eqref{c1} that
\ben
\G_r(\lam)_{r-1,j}=\left(\begin{array}{cccccc|cc}
  & & & & & & & \\
 &&\mathbb{H}_{*,j}&&& &b_{r-3}& 0 \\
 & &&  & & &-1&b_{r-2}\\\cline{1-6}
\multicolumn{6}{c}{ }&&\\
c_1&\cdots&c_{j-1}&c_{j+1}&\cdots &\multicolumn{1}{c}{c_{r-2}}&c_{r-1}&c_{r}
\end{array}\right).
\een
Thus by expanding the determinant by the last column, we have by using \eqref{c3}, \eqref{c5} and the definition of $\G_{r-1}(\lam)$ that
\ben
& &\det[\G_r(\lam)_{r-1,j}]\\
&=&c_r\det[\E_{r+1}(\lam)_{(r-1,r,r+1),(j,r,r+1)}]-b_{r-2}\det[\E_{r+1}(\lam)_{(r-2,r-1,r),(j,r,r+1)}]\\
&=&\det[\E_r(\lam)_{r-1,j}]-b_{r-2}\det[\E_r(\lam)_{(r-2,r-1),(j,r)}]\\
&=&\det[\E_r(\lam)_{r-1,j}]-b_{r-2}\det[\G_{r-1}(\lam)_{r-2,j}].
\een
This completes the proof.  
\end{proof}

The minors of $\E_{r+1}(\lam)$ for $r=1,2$ can be easily computed from \eqref{y6}. The following theorem gives the recursive formulas for computing all minors of $\E_{r+1}(\lam)$ for $r\ge 3$.

\begin{thm}\label{thm:3.2}
Let $r\ge 3$, we have \\
$1^\circ$ For $i=1,\cdots,r-2$,
\ben
& &\det[\E_{r+1}(\lam)_{i,j}]=\det[\E_{r}(\lam)_{i,j}]+a_ra_{r-1}\det[\E_{r-1}(\lam)_{i,j}],\ \ 1\le j\le r-2,\\
& &\det[\E_{r+1}(\lam)_{i,j}]=(-1)^{j-i-1}\frac{a_1\cdots a_{j-1}}{a_1\cdots a_i}\det\G_i(\lam),\ \ j=r-1,r,r+1.
\een
$2^\circ$ For $i=r-1$, we have
\ben
& &\det[\E_{r+1}(\lam)_{r-1,j}]=-c_{r+1}a_r\det[\E_r(\lam)_{r,j}]+\det[\G_r(\lam)_{r-1,j}],\ \ 1\le j\le  r-1,\\
& &\det[\E_{r+1}(\lam)_{r-1,j}]=(-1)^{j-i-1}\frac{a_1\cdots a_{j-1}}{a_1\cdots a_i}\det\G_i(\lam),\ \ j=r,r+1.
\een
$3^\circ$ For $i=r$, we have
\ben
& &\det[\E_{r+1}(\lam)_{r,j}]=\det[\E_{r}(\lam)_{r-1,j}]-a_{r-1}b_{r-2}\det[\E_{r-1}(\lam)_{r-2,j}],\ \ 1\le j\le  r-2,\\
& &\det[\E_{r+1}(\lam)_{r,j}]=c_{r+1}(-1)^{i-j}\prod^{j-1}_{k=1}a_k,\ \ j=r-1,r,\\
& &\det[\E_{r+1}(\lam)_{r,r+1}]=\det\G_{r-1}(\lam)-a_{r-1}b_{r-2}\det\G_{r-2}(\lam)+c_r\prod^{r-1}_{k=1}a_k.
\een
$4^\circ$ For $i=r+1$, we have
\ben
& &\det[\E_{r+1}(\lam)_{r+1,j}]=-\det[\E_{r}(\lam)_{r,j}]-a_{r-1}b_{r-2}\det[\E_{r-1}(\lam)_{r-1,j}],\ \ 1\le j\le r-2,\\
& &\det[\E_{r+1}(\lam)_{r+1,j}]=(-1)^{i-j}\prod^{j-1}_{k=1}a_k,\ \ j=r-1,r,r+1.
\een
\end{thm}
\begin{proof} The proof is divided into 4 steps. \\ {\sc Step 1}. The first equality in $1^\circ$ can be proved by the same argument as that in Lemma \ref{lem:3.1}. Here we omit the details. We only prove the second equality in $1^\circ$ when $j=r+1$. The other cases can be proved similarly.
By the partition in \eqref{c1}, we know that for $1\le i\le r-2$,
\ben
\E_{r+1}(\lam)_{i,r+1}=\left(\begin{array}{ccc|cc}
  & & & & \\
 &\mathbb{H}_{i,*}& &\ \ \ \ \ \mathbb{F}_{i,*}& \hskip-2mm\\
 & & & &\\\cline{1-3}\cline{4-5}
\multicolumn{3}{c}{}&a_{r-1}&-1\\
\multicolumn{3}{c}{} &0&a_{r}\\
c_1&\cdots &\multicolumn{1}{c}{c_{r-2}}&c_{r-1}&c_{r}
\end{array}\right),\ \ \ \
\mathbb{F}=\left(\begin{array}{ccc}
0 & 0& 0\\
\vdots &\ \vdots &\vdots\\
0 & 0 & 0\\
b_{r-3} &0 &0\\
-1&b_{r-2} & 0\\
\end{array}\right),
\een
where $\mathbb{F}\in\R^{(r-2)\times 3}$. By expanding the determinant first by the $r$-th and then by the $(r-1)$-th row, we know by \eqref{c3} that
\ben
\det[\E_{r+1}(\lam)_{i,r+1}]&=&a_ra_{r-1}\det[\E_{r+1}(\lam)_{(i,r-1,r),(r-1,r,r+1)}]\\
&=&a_ra_{r-1}\det[\E_{r}(\lam)_{(i,r-1),(r-1,r)}]\\
&=&a_ra_{r-1}\det[\G_{r-1}(\lam)_{i,r-1}].
\een
This shows the second equality in $1^\circ$ by the first identity in Lemma \ref{lem:3.3}.

{\sc Step 2}. We only prove the first equality in $2^\circ$ when $1\le j\le r-2$. The other cases can be proved similarly. By the partition in \eqref{c1}, we have for $1\le j\le r-2$,
\ben
\E_{r+1}(\lam)_{r-1,j}=\left(\begin{array}{ccc|ccc}
  & & & & & \\
 & \mathbb{H}_{*,j}& &b_{r-3}& 0&0 \\
 & & &-1&b_{r-2}&0\\\cline{1-3}
\multicolumn{3}{c}{ }&0&a_{r}&-1\\
c_1&\cdots &\multicolumn{1}{c}{c_{r-2}}&c_{r-1}&c_{r}&c_{r+1}
\end{array}\right)
\een
Expanding the determinant by the last column, we obtain
\ben
\det[\E_{r+1}(\lam)_{r-1,j}]
&=&c_{r+1}a_r\det[\E_{r+1}(\lam)_{(r-1,r,r+1),(j,r,r+1)}]+\det[\E_{r+1}(\lam)_{(r-1,r),(j,r+1)}]\\
&=&-c_{r+1}a_r\det[\E_r(\lam)_{r,j}]+\det[\G_{r}(\lam)_{r-1,j}].
\een
This shows the first equality in $2^\circ$ for $1\le j\le r-2$.

{\sc Step 3.} The second equality in $3^\circ$ can be easily proved. The last equality is shown in Lemma \ref{lem:3.2} since $\E_{r+1}(\lam)_{r,r+1}=\G_r(\lam)$ by definition. To show the first equality in $3^\circ$, we again use the partition in \eqref{c1} to obtain for $1\le j\le r-2$,
\ben
\E_{r+1}(\lam)_{r,j}=\left(\begin{array}{cccccc|ccc}
  & && & & & & &\\
 &&\mathbb{H}_{*,j}&&& &b_{r-3}& 0&0 \\
 & &&  & & &-1&b_{r-2}&0\\\cline{1-6}
\multicolumn{6}{c}{ }&a_{r-1}&-1&0\\
c_1&\cdots&c_{j-1}&c_{j+1}&\cdots &\multicolumn{1}{c}{c_{r-2}}&c_{r-1}&c_{r}&c_{r+1}
\end{array}\right).
\een
By expanding the determinant successively by the last columns, we have
\ben
& &\det[\E_{r+1}(\lam)_{r,j}]\\
&=&-c_{r+1}\det[\E_{r+1}(\lam)_{(r-1,r,r+1),(j,r,r+1}]-c_{r+1}b_{r-2}\det[\E_{r+1}(\lam)_{(r-2,r,r+1),(j,r,r+1)}].
\een
This shows the first equality in $3^\circ$ by Lemma \ref{lem:3.0}.

{\sc Step 4.} The first equality in $4^\circ$ can be proved by the same argument as that Step 3. The second equality can be easily proved. Here we omit the details.  
\end{proof}

This theorem indicates that all minors of $\E_{r+1}(\lam)$ can be computed once one knows all minors of $\E_m(\lam)$, $1\le m\le r$, and the minors $\det[\G_r(\lam)_{r-1,j}]$, $1\le j\le r-1$, which can be computed by Lemma \ref{lem:3.3} recursively based on the information of the minors $E_m(\lam)$, $1\le m\le r$.

The following lemma indicates that the nodal values of the solution to \eqref{dp} depends only on the coefficient $\bfa_0$.

\begin{lem}\label{lem:3.4}
Let $\bY_r(t)$, $r\ge 1$,  be the solution of the problem \eqref{dp}. Then
\begin{align*}
\bY_r(t_{n+1})=\bY_r(t_n)+\tau_n \D \bfa_0+\tau_n \bR_0,\ \ \ n=1,\cdots, N-1.
\end{align*}
\end{lem}
\begin{proof}
We integrate \eqref{dp} over $I_n$ and use the orthogonality of Legendre polynomials to obtain
\begin{align*}
\bY_r(t_{n+1})&=\bY_r(t_n)+\int_{I_n}\D\,\left(\sum_{j=0}^{r-1}\bfa_j\widetilde{L}_j(t)\right)\, dt+\int_{I_n} \sum_{j=0}^{r-1}\bR_j\widetilde{L}_j(t)\, dt,\\
&=\bY_r(t_n)+\tau_n \D \bfa_0+\tau_n \bR_0.
\end{align*}
This completes proof.  
\end{proof}

By Theorem \ref{thm:3.1}, we have then
\ben
\bY_r(t_{n+1})=\bY_r(t_n)+\sum^r_{j=1}(\zeta_j\I+\tau_n\D)^{-1}\left(\sum^{r+1}_{k=1}(-1)^k\frac{\phi_{k1}(-\zeta_j)}{P_r'(\zeta_j)}(\tau_n\D)\,\bfb_{k-1}\right)+\tau_n\bR_0.
\een
This leads to the following parallel algorithm to compute the nodal values of the solution $\bY_r$ to the problem \eqref{dp}.

\medskip
\begin{alg}\label{aa1} Given $\bY_r(t_0)=\bY_0$. For $n=1,\cdots, N-1$, do the following.\\
$1^\circ$ Compute $\bv_j\in\R^M$, $j=1,\cdots,r$, in parallel, where
 \begin{align*}
 \bv_j=\sum^{r+1}_{k=1}(-1)^k\frac{\phi_{k1}(-\zeta_j)}{P_r'(\zeta_j)}(\tau_n\D)\,\bfb_{k-1}.
 \end{align*}
$2^\circ$ Solve $(\tau_n \D+\zeta_j\I)\mathbf{w}_{j}=\bv_j$, $j=1,\cdots,r$, in parallel.\\
$3^{\circ}$ Compute
  \begin{align*}
  \bY_r(t_{n+1})=\bY_r(t_n)+\sum_{j=0}^{r}\mathbf{w}_j+\tau_n \bR_0.
  \end{align*}
\end{alg}

The following parallel-in-time algorithm computes the solution of the problem \eqref{dp} inside each time interval.

\medskip
\begin{alg}\label{aa2} Given $\bY_r(t_0)=\bY_0$. \\
$1^\circ$ Call Algorithm \ref{aa1} to obtain $\bY_r(t_n)$, $n=1,\ldots,N$.\\
$2^\circ$ Compute the coefficients $\bfa_1,\cdots,\bfa_r$ of $\bY_r$ in each time interval $I_n$, $n=1,\cdots,N-1$, in parallel as follows. \\
{\rm (i)} Compute $\bv_{ij}\in\R^M$, $i=2,\cdots,r+1$, $j=1,\cdots r$, in parallel, where
\ben
\bv_{ij}=\sum^{r+1}_{k=1}(-1)^{i+k+1}\frac{\phi_{ki}(-\zeta_j)}{P_r'(\zeta_j)}\,\bfb_{k-1}.
\een
{\rm (ii)} Solve $(\tau_n\D+\zeta_j\I)\mathbf{w}_{ij}=\bv_{ij}$, $i=2,\cdots,r+1,j=1,\cdots,r$, in parallel.\\
{\rm (iii)} Compute $\bfa_{i-1}=(-1)^r\delta_{i,r+1}\bfb_{r}+\sum_{j=1}^r\mathbf{w}_{ij}$, $i=2,\cdots,r+1$, in parallel.
\end{alg}
\medskip

\begin{rem}\label{rem1}
In \cite{Saad1989}, it is observed that the zeros of $P_r(z)$ come in complex conjugate pairs if they are complex. If $\zeta_{j'}=\bar\zeta_j$, $j,j'=1,\cdots,r$, then $\bv_j=\bar{\bv}_{j'}$, and
\ben
\bw_j+\bw_{j'}=\frac{\bv_j}{\tau_n\D+\zeta_j\I}+\frac{\bar\bv_j}{\tau_n\D+\bar\zeta_j\I}=2\,\Re\,\left[\frac{\bv_j}{\tau_n\D+\zeta_j\I}\right].
\een
Thus one need only to solve $k$ complex matrix problems instead of $2k$ in Algorithm \ref{aa1} {\rm ($2^\circ$)} and $kr$ complex matrix problems instead of $2kr$ in Algorithm \ref{aa2} {\rm ($2^\circ$)}, where $2k$, $0\le k\le r/2$, are the number of complex zeros of $P_r(z)$.
\end{rem}

\begin{rem}\label{rem2}
If $\zeta=a+\i b$ is a complex zero of $P_r(z)$, then $a\le -2$ by Lemma \ref{lem:2.4}. By Remark \ref{rem1}, without loss of generality, we can choose one of the zeros such that $b<0$. Let $\bw=\bw_1+\i\bw_2$, $\bv=\bv_1+\i\bv_2$, where $\bv_i,\bw_i\in\R^M$, $i=1,2$, satisfy $(\tau_n\D+\zeta\I)\bw=\bv$. Then
\ben
\widetilde{\D}\left(\begin{array}{c}
\bw_1 \\
\bw_2 \\
\end{array}\right):=\left(\begin{array}{cc}
\tau_n\D+a\I & -b\,\I\\
-b\,\I & -(\tau_n\D+a\I)\\
\end{array}\right)\left(\begin{array}{c}
\bw_1 \\
\bw_2 \\
\end{array}\right)=\left(\begin{array}{c}
\bv_1 \\
-\bv_2 \\
\end{array}\right).
\een
Let $\mathbb{F}={\rm diag}(\tau_n\D+(a+b)\,\I,\tau_n\D+(a+b)\,\I)\in\R^{2M\times 2M}$ be the diagonal matrix. It is shown in Chen et al \cite[Lemma 4.1]{Chen} that the condition number $\kappa(\mathbb{F}^{-1}\widetilde{\D})\le\sqrt 2$. Therefore, the complex system $(\tau_n\D+\zeta\I)\bw=\bv$ can be efficiently solved if one has the efficient solver for the real matrix
$\tau_n\D+(a+b)\,\I$, where $a+b\le -2$. Notice that the eigenvalues of $\D$ lie in the left half-plane due to the assumption $\D+\D^T\le 0$.
\end{rem}

\begin{rem} \label{rem3}
For non-standard ODE systems of the form,
\begin{align}\label{equ_ODE_nonstandard}
\bM\bY'=\D\bY+\bR\ \ \text{ in } (0,T), \quad \bY(0)=\bY_0,
\end{align}
one can use the transformations $\widetilde{\bY}=\bM^{\frac 12}\bY$, $\widetilde{\D}=\bM^{-\frac 12}\D\bM^{-\frac 12}$, $\widetilde{\bR}=\bM^{-\frac 12}\bR$ to transform the problem \eqref{equ_ODE_nonstandard} to \eqref{p1} and use above algorithms to solve the transformed problem. This leads to the following algorithm which is similar to Algorithm \ref{aa1} to find the nodal values of the solution of the continuous time Galerkin method for solving \eqref{equ_ODE_nonstandard}. A similar algorithm to Algorithm \ref{aa2} can also be formulated.

\medskip
\begin{alg}\label{aa4} Given $\bY_r(t_0)=\bY_0$. For $n=1,\cdots, N-1$, do the following.\\
$1^\circ$ Compute $\bv_j\in\R^M$, $j=1,\cdots,r$, in parallel, where
 \begin{align*}
 \bv_j=\sum^{r+1}_{k=1}(-1)^k\frac{\phi_{k1}(-\zeta_j)}{P_r'(\zeta_j)}(\tau_n\D)\mathbb{M}^{-1}\,\bfb_{k-1}.
 \end{align*}
$2^\circ$ Solve $(\tau_n \D+\zeta_j\mathbb{M})\mathbf{w}_{j}=\bv_j$, $j=1,\cdots,r$, in parallel.\\
$3^{\circ}$ Compute
  \begin{align*}
  \bY_r(t_{n+1})=\bY_r(t_n)+\sum_{j=1}^{r}\mathbf{w}_j+\tau_n\mathbb{M}^{-1}\bR_0.
  \end{align*}
\end{alg}
\end{rem}

To conclude this section, we prove the following theorem for finding the nodal values of the solution \eqref{dp} which extends \eqref{a1} for solving the ODE system \eqref{p1} when $\bR=\bf{0}$.

\begin{thm}\label{thm:3.3}
Let $\bY_r\in\bV^r_\tau$, $r\ge 1$, be the solution of the problem \eqref{dp}. Then for $n=1,\cdots,N-1$,
\ben
\bY_r(t_{n+1})=\frac{P_r(\tau_n\D)}{P_r(-\tau_n\D)}\bY_r(t_n)+\sum^r_{k=1}(-1)^{k+1}\frac{\phi_{k1}(\tau_n\D)}{P_r(-\tau_n\D)}\,\bfb_{k-1}+\tau_n\bR_0,
\een
where $\phi_{k1}(\lam)=\det[\E_{r+1}(\lam)_{k,1}]$.
\end{thm}

\begin{proof}
By Lemma \ref{lem:3.1} and \eqref{y7} we have
\ben
P_r(-\tau_n\D)\bfa_0=\sum^{r+1}_{k=1}(-1)^{k+1}\phi_{k1}(\tau_n\D)\bfb_{k-1}.
\een
Since $\bfb_{r+1}=\bY_r(t_n)$, by Lemma \ref{lem:3.4},
\ben
\bY_r(t_{n+1})&=&\left[\I+(-1)^r\frac{\tau_n\D\phi_{r+1,1}(\tau_n\D)}{P_r(-\tau_n\D)}\right]\bY_r(t_n)\\
& &+\sum^r_{k=1}(-1)^{k+1}\phi_{k1}(\tau_n\D)\bfb_{k-1}+\tau_n\bR_0.
\een
Denote by $\psi_r(\lam)=(-1)^r\lam\phi_{r+1,1}(\lam)=(-1)^r\lam\det[\E_{r+1}(\lam)_{r+1,1}]$. By Theorem \ref{thm:3.2}, $4^\circ$, we know that $\psi_r(\lam)$ satsifies
\ben
\psi_r(\lam)=\psi_{r-1}(\lam)+\frac{\lam^2}4\frac 1{(2r-1)(2r-3)}\psi_{r-2}(\lam),\ \ r\ge 3.
\een
On the other hand, by \eqref{y6}, we have $\psi_1(\lam)=\lam,\psi_2(\lam)=\lam$. This implies by Lemma \ref{lem:2.4} that $\psi_r(\lam)=P_r(\lam)-P_r(-\lam)$. Thus
\ben
\I+(-1)^r\frac{\tau_n\D\phi_{r+1,1}(\tau_n\D)}{P_r(-\tau_n\D)}=\I+\frac{P_r(\tau_n\D)-P_r(-\tau_n\D)}{P_r(-\tau_n\D)}=P_r(\tau_n\D).
\een
This completes the proof.  
\end{proof}

\section{The dissipative system}
In this section, we propose an alternative way to compute the coefficients $\bfa_1,\cdots,\bfa_r$ of the solution $\bY_r$ of the problem \eqref{dp} when the ODE system \eqref{p1} is dissipative $\D+\D^T<0$. The algorithm is based on the block tridiagonal structure of the matrix and is less expensive than the step $2^\circ$ in Algorithm \ref{aa2}.

Let $\widetilde{\mathbf{X}}=(\bfa_1^T,\cdots,\bfa_r^T)^T\in\R^{rM}$ and $\widetilde{\bB}=(\widetilde{\bfb}^T_1,\bfb_2^T,\cdots,\bfb_r^T)^T$ with $\widetilde{\bfb}_1=\bfb_1-\frac 12(\tau_n\D)\bfa_0$. It follows from \eqref{y2} that $\widetilde{\mathbf{X}}\in\R^{rM}$ satisfies
\beq\label{z1}
\widetilde{\A}\widetilde{\mathbf{X}}=\widetilde{\bB},
\eeq
where
\ben
\widetilde{\A}=\left(\begin{array}{cccccc}
-\I&-\frac{\tau_n}{2}\D\frac{1}{5}&\cdots&\cdots&\cdots&0\\
\frac{\tau_n}{2}\D\frac{1}{3}&-\I&-\frac{\tau_n}{2}\D\frac{1}{7}&\cdots&\cdots&0\\
\vdots&\ddots&\ddots&\ddots&\vdots&\vdots\\
\cdots&\cdots&\frac{\tau_n}{2}\D\frac{1}{2r-5}&-\I&-\frac{\tau_n}{2}\D\frac{1}{2r-1}&0\\
\cdots&\cdots&\cdots&\frac{\tau_n}{2}\D\frac{1}{2r-3}&-\I&0\\
\cdots&\cdots&\cdots&\cdots&\frac{\tau_n}{2}\D\frac{1}{2r-1}&-\I\\
\end{array}\right).
\een
It is easy to see that $\widetilde{\A}=\E_{r+1}(\tau_n\D)_{r+1,1}$, let $\phi_r(\lam)=\det[\E_{r+1}(\lam)_{r+1,1}]$. The goal of this section is to show that $\phi_r(\tau_n\D)\in\R^{M\times M}$ is invertible so that the standard chasing algorithm for block tridiagonal matrices (cf., e.g., Golub and Van Load \cite[\S4.5]{Golub}) can be used to solve \eqref{z1}.

We start with the following lemma.

\begin{lem}\label{lem:4.1}
Let $u_n\in P^n$, $n\ge 0$, such that {\rm (i)} $u_0=0, u_1=1$ or {\rm (ii)} $u_0=1,u_1=1+A_1t$, $A_1\in\R$, and for $n\ge 2$,
\beq
u_n(t)=(1+A_nt)u_{n-1}(t)-C_nt^2u_{n-2},\label{z2}
\eeq
where $A_n\in\R, C_n>0$ for $n\ge 2$. Then we have, for $n\ge 2$,
\ben
u_n(u_{n-1}+tu'_{n-1})-tu_{n-1}u'_n>0\ \ \ \mbox{in }\R\backslash\{0\}.
\een
\end{lem}
\begin{proof} We set, for $n\ge1$,
\ben
G_n(t,t')=\frac{t'u_n(t)u_{n-1}(t')-tu_{n-1}(t)u_n(t')}{t'-t}.
\een
It is easy to see by \eqref{z2} that $G_1(t,t')=u_0$, and for $n\ge 2$,
\ben
G_n(t,t')=u_{n-1}(t)u_{n-1}(t')+C_ntt'G_{n-1}(t,t').
\een
This implies easily
\ben
G_n(t,t')=u_{n-1}(t)u_{n-1}(t')+\sum^{n-2}_{k=0}C_n\cdots C_{k+2}(tt')^{n-1-k}u_k(t)u_k(t'),
\een
where we have used $G_1(t,t')=u_0(t)u_0(t')$ in both cases (i) and (ii).
By letting $t'\to t$, we obtain for $n\ge 2$,
\ben
\lim_{t'\to t}G_n(t,t')=u_n(t)[tu_{n-1}(t)]'-tu_{n-1}(t)u_n'(t)>0\ \ \mbox{in }\R\backslash\{0\},
\een
where we have used the condition $u_1=1$ in the case (i) and $u_0=1$ in the case (ii). This completes the proof.  
\end{proof}

The following theorem is the main result of this section.

\begin{thm}\label{thm:4.1}
Let $\D\in\R^{M\times M}$ satisfy $\D+\D^T<0$, that is, $\D+\D^T$ is negative definite. Then the matrix $\phi_r(\tau_n\D)\in\R^{M\times M}$, $r\ge 1$, $1\le n\le N-1$, is invertible.
\end{thm}

\begin{proof} We are going to show that all zeros of $\phi_r(\lam)$ locate at the imaginary axis $\{z\in\C:\Re(z)=0\}$. This implies easily $\phi_r(\tau_n\D)$ is invertible since the eigenvalues of $\D$ lie in the left-half plane due to the dissipative property $\D+\D^T<0$.

To study the property of the zeros of $\phi_r(\lam)$, we denote $\psi_r(\lam)=(-1)^r\phi_r(\lam)$. By \eqref{y6} and Theorem \ref{thm:3.2} ($4^\circ$) we know that $\psi_1(\lam)=1,\psi_2(\lam)=1$, and
\beq\label{z3}
\psi_r(\lam)=\psi_{r-1}(\lam)+d_{r}\lam^2\psi_{r-2}(\lam),\ \ \ d_r=\frac 1{4(2r-1)(2r-3)},\ \ \ \mbox{for }r\ge 3.
\eeq
We note that \eqref{z3} is also valid for $r=2$ if we define $\psi_0(\lam)=0$. Set $t=\lam^2$ and for $m\ge 0$, define $f_m(t)=\psi_{2m}(\lam), g_m=\psi_{2m+1}(\lam)$. Then
\ben
& &f_m(t)=g_{m-1}(t)+d_{2m}tf_{m-1}(t),\ \ f_0(t)=0,f_1(t)=1,\\
& &g_m(t)=f_{m}(t)+d_{2m+1}tg_{m-1}(t),\ \ g_0(t)=1,g_1(t)=1+d_3t.
\een
This implies that, for $m\ge 2$,
\begin{align}
&f_m(t)=(1+A_mt)f_{m-1}(t)-C_mt^2f_{m-2}(t),\label{z4}\\
&g_m(t)=(1+\tilde{A}_mt)g_{m-1}(t)-\tilde{C}_mt^2f_{m-2}(t),\label{z6}
\end{align}
where $A_m=d_{2m}+d_{2m-1}, C_m=d_{2m-1}d_{2m-2}, \tilde{A}_m=d_{2m+1}+d_{2m}, \tilde{C}_m=d_{2m}d_{2m-1}$.
We observe that $f_m,g_m$ satisfy the same recurrence relation but with different coefficients and initial values. In the following we will only prove $f_m$, $m\ge 2$, has $m$ real zeros in $(-\infty,0)$ which then implies that $\phi_{2m}(\lam)=f_m(\lam^{1/2})$ has all zeros on the imaginary axis. The proof for $\phi_{2m+1}(\lam)$, $m\ge 1$, is similar and we omit the details.

We extend the argument in \cite[\S3.3 (4)]{Szego} for orthogonal polynomials to show that $f_m, m\ge 2$, has $m$ zeros in $(-\infty,0)$ by using Sturm theorem (cf., e.g., Perron \cite[pp.7-9]{Perron}) based on the recurrence formula \eqref{z4}. We first note that if $f_m(t)=\sum^m_{k=0}\theta_kt^k$, then $\theta_m>0$ since by \eqref{z3} the leading coefficients of the polynomial $\psi_r(\lam)$ are positive, and $\theta_0=f_m(0)=1$ by \eqref{z4}. Now we claim that
\beq\label{z5}
f_m(t),f_{m-1}(t),\cdots,f_1(t)
\eeq
form a Sturmian sequence in $[-M,-\delta]$ for sufficiently large $M>0$ and sufficiently small $\delta>0$ in the following sense. (i) $f_1(t)=1$ has no zeros in $[-M,-\delta]$. (ii) $f_m(-M)f_m(-\delta)\not=0$ for $M\gg 1$ and $\delta\ll 1$ since $\theta_m>0$ and $\theta_0=1$. (iii) If $c\in [-M,-\delta]$ is a zero of $f_k(t)$, $1\le k\le m-1$, then $f_{k+1}(c)f_{k-1}(c)<0$. In fact, By Lemma \ref{lem:4.1}, $f_{k-1}(c)\not=0$. By \eqref{z4}, $f_{k+1}(c)=-C_mc^2f_{k-1}(c)$. (iv) If $c\in [-M,-\delta]$ such that $f_m(c)=0$, then $f_m'(c)f_{m-1}(c)>0$, which is a direct consequence of Lemma \ref{lem:4.1}. Now the number of variations of sign in \eqref{z5} at $t=-M$ is $m$ for sufficiently large $M$ since $\theta_m>0$; it is zero at $t=-\delta$ for sufficiently small $\delta>0$ since $\theta_0=1$. Thus by Sturm theorem we conclude that $f_m(t)$ has exactly $m$ zeros in $[-M,-\delta]$. This completes the proof.  
\end{proof}

Based on this theorem, we can use the following parallel-in-time algorithm to compute the solution $\bY_r$ to the problem \eqref{dp}.

\medskip
\begin{alg}\label{aa3} Given $\bY_r(t_0)=\bY_0$. \\
$1^\circ$ Call Algorithm \ref{aa1} to obtain $\bY_r(t_n)$ for $n=1,\ldots,N$.\\
$2^\circ$ Compute $\bY_r$ in each time interval $I_n$, $n=1,\cdots,N-1$, in parallel by solving \eqref{z1} to obtain $\bfa_1,\cdots,\bfa_r$ using the LU decomposition for block tridiagonal matrices.
\end{alg}
\medskip

We remark that the algorithm of $LU$ decomposition for block tridiagonal matrices for solving \eqref{z1} requires to solve $r$ systems of linear equations of size $M$ in sequential instead of to solve $r^2$ systems of linear equations of size $M$ in parallel in Step $2^\circ$ of Algorithm \ref{aa2}.

\section{Optimal stability and error esstimates}\label{sec_hyp}
In this section, we show optimal stability and error estimates of the continuous time Galerkin method \eqref{dp} in terms of $r$ when $\D$ is a symmetric or skew-symmetric matrix. This will be achieved by using the explicit formulas in Theorem \ref{thm:3.1}. We start by studying further properties of the minors of the stiffness matrix of the continuous time Galerkin method $\mathbb{A}=\E_{r+1}(\tau_n\D)$.

Let $\chi_{r+1,j}(\lam)=(-1)^{r+1}\det[\E_{r+1}(\lam)_{r+1,j}]$, then by \eqref{y6} we have
\beq\label{k1}
\chi_{2,1}=-1,\chi_{2,2}=a_1,\chi_{3,1}=-1,\chi_{3,2}=a_1,\chi_{3,3}=-a_1a_2.
\eeq
For $r\ge 3$, by $4^{\circ}$ in Theorem \ref{thm:3.2}, we have the following recursive formulas
\begin{align}
&\chi_{r+1,j}(\lam)=\chi_{r,j}(\lam)+a_{r}a_{r-1}\chi_{r-1,j}(\lam),\ \ 1\le j\le r-2,\label{k2}\\
&\chi_{r+1,j}(\lam)=(-1)^{j}\prod^{j-1}_{k=1}a_k,\ \ j=r-1,r,r+1.\label{k3}
\end{align}
Let $\varphi_0=1$ and $\varphi_r(\lam)=\det\E_{r+1}(\lam)$, $r\ge 1$. Then by Lemma \ref{lem:3.1}, $\varphi_1=1-a_1$,
\be\label{k4}
\varphi_{r+1}=\varphi_r+a_ra_{r+1}\varphi_{r-1},\ \ r\ge 1.
\ee

\begin{lem}\label{lemma_app1}
For any $r\ge1$ and $\lam\in\R$, we have
\begin{align}
&\sum_{j=1}^{r-1}[\chi_{r,j}(-\lam)\chi_{r+1,j}(\lam)+\chi_{r,j}(\lam)\chi_{r+1,j}(-\lam)]\frac{1}{2j-1}\nn\\
&=\varphi_{r-1}(-\lam)\varphi_r(\lam)+\varphi_{r-1}(\lam)\varphi_r(-\lam)+\frac{(-1)^r2r}{2r-1}(a_1\cdots a_{r-1})^2,\label{eqn_app1}\\
&\sum_{j=1}^{r+1}\chi_{r+1,j}(-\lam)\chi_{r+1,j}(\lam)\frac{1}{2j-1}=\varphi_{r}(-\lam)\varphi_r(\lam)+\frac{(-1)^{r+1}2r}{2r+1}(a_1\cdots a_r)^2.\label{eqn_app2}
\end{align}
\end{lem}

\begin{proof}
We denote
\ben
A_{r}:&=&\sum_{j=1}^{r-1}\left[\chi_{r,j}(-\lam)\chi_{r+1,j}(\lam)+\chi_{r,j}(\lam)\chi_{r+1,j}(-\lam)\right]\frac{1}{2j-1}\\
& &-\left[\varphi_{r-1}(-\lam)\varphi_r(\lam)+\varphi_{r-1}(\lam)\varphi_r(-\lam)\right],\\
B_{r+1}:&=&\sum_{j=1}^{r+1}\chi_{r+1,j}(-\lam)\chi_{r+1,j}(\lam)\frac{1}{2j-1}-\varphi_{r}(-\lam)\varphi_r(\lam).
\een
We will argue by induction. First \eqref{eqn_app1}-\eqref{eqn_app2} are obvious for $r=1,2$ by \eqref{k1}. Now we assume \eqref{eqn_app1}-\eqref{eqn_app2} are valid for all $r\leq n$, $n\ge 2$. Since by \eqref{k3}, $\chi_{n+2,n}(\lam)=\chi_{n+1,n}(\lam)$, we have by \eqref{k2} and \eqref{k4} that
\ben
A_{n+1}=2B_{n+1}-2\chi_{n+1,n+1}(-\lam)\chi_{n+1,n+1}(\lam)\frac{1}{2n+1}+a_na_{n+1}A_n.
\een
Now by \eqref{k3} and the induction assumption that \eqref{eqn_app1}-\eqref{eqn_app2} are valid for $r=n$, we obtain
\ben
A_{n+1}=\frac{(-1)^{n+1}2(n+1)}{2n+1}(a_1\cdots a_n)^2.
\een
This shows \eqref{eqn_app1} for $r=n+1$. Similarly, we can prove by \eqref{k2}-\eqref{k4} that 
\ben
B_{n+2}&=&B_{n+1}+(a_na_{n+1})^2B_n+a_na_{n+1}A_n\\
& &+\chi_{n+2,n+2}(-\lam)\chi_{n+2,n+2}(\lam)\frac 1{2n+3}-(a_na_{n+1})^2\chi_{n,n}(-\lam)\chi_{n,n}(\lam)\frac{1}{2n-1}.
\een
Now by the induction assumption \eqref{eqn_app2} for $r=n,n+1$ and \eqref{eqn_app1} for $r=n$, we obtain by using \eqref{k3} that
\ben
B_{n+2}=(-1)^{n+2}\frac{2(n+1)}{2n+3}(a_1\cdots a_{n+1})^2.
\een
This completes the proof.  
\end{proof}

\begin{lem}\label{lem:5.2}
Let $r\ge 1$. For any $\lam\le 0$, we have
\beq\label{k5}
\sum^{r-1}_{j=1}\chi_{r,j}(\lam)\chi_{r+1,j}(\lam)\frac 1{2j-1}\le\varphi_{r-1}(\lam)\varphi_r(\lam),\ \ \sum^{r+1}_{j=1}\chi_{r+1,j}(\lam)^2\frac 1{2j-1}\le\varphi_r(\lam)^2.
\eeq
\end{lem}

\begin{proof} For any $r\ge 1$, we denote
\ben
C_r:=\sum^{r-1}_{j=1}\chi_{r,j}\chi_{r+1,j}\frac 1{2j-1}-\varphi_{r-1}\varphi_r,\ \ D_{r+1}:=\sum^{r+1}_{j=1}\chi^2_{r+1,j}\frac 1{2j-1}-\varphi^2_r.
\een
We again argue by induction. First \eqref{k5} is obvious for $r=1,2$ by \eqref{k1} since $\lam\le 0$. Now we assume \eqref{k5} is valid for all $r\leq n$, $n\ge 2$.
By \eqref{k2}-\eqref{k4}, it is easy to see that
\ben
C_{n+1}=-\chi^2_{n+1,n+1}\frac 1{2n+1}+D_{n+1}+a_na_{n+1}C_n,
\een
where we have used $\chi_{n+2,n}(\lam)=\chi_{n+1,n}(\lam)$. Thus if $C_n\le 0, D_{n+1}\le 0$, then $C_{n+1}\le 0$.

On the other hand, by \eqref{k2}-\eqref{k4}, we have
\ben
D_{n+2}=D_{n+1}+a_na_{n+1}D_n+2a_na_{n+1}C_n+(a_1\cdots a_{n+1})^2(\frac 1{2n+3}-\frac 1{2n-1}).
\een
Thus $D_{n+2}\le 0$ if $D_n\le 0,D_{n+1}\le 0$, and $C_n\le 0$. This completes the proof.  
\end{proof}

The following theorem is the main result of this section.

\begin{thm}\label{thm:5.1}
Let $\D$ be a symmetric or skew-symmetric matrix and $\bY_r\in\bV^r_\tau$ is the solution of the problem \eqref{dp}. Then we have
\begin{align}
&\|\bY_r\|_{L^2(0,T)}\leq T^{1/2}\|\bY_0\|_{\mathbb{R}^M}+CT\|\bR\|_{L^2(0,T)},\label{stab_L2}\\
&\max_{0\leq t\leq T}\|\bY_r\|_{\mathbb{R}^M}\leq Cr(\|\bY_0\|_{\mathbb{R}^M}+T^{1/2}\|\bR\|_{L^2(0,T)}),\label{stab_Linf}
\end{align}
where the constant $C$ is independent of $\tau,r,\D$, and $\bR$.
\end{thm}

\begin{proof}
Let $\hat\bY_r\in [P^r]^M$ be defined in \eqref{g1} of Lemma \ref{lem:2.2}, we claim that
\be\label{k6}
\|\hat\bY_r\|_{L^2(I_n)}\le\tau_n^{1/2}\|\bY^n_r\|_{\R^M},
\ee
which improves the bound \eqref{g6} in the proof of Lemma \ref{lem:2.2}. To show \eqref{g1}, by Theorem \ref{thm:3.1}, we have
\begin{align*}
\hat{\bY}_r(t)&=\sum_{j=1}^{r+1}\bfa_{j-1}\tilde{L}_{j-1}(t)
=\sum_{j=1}^{r+1}(-1)^{j}\frac{\chi_{r+1,j}(\tau_n\D)}{\varphi_r(\tau_n\D)}\bY_r^n\tilde{L}_{j-1}(t)\ \ \forall t\in I_n.
\end{align*}
Thus by \eqref{d2},
\be\label{k7}
\|\hat{\bY}_r\|_{L^2(I_n)}^2=\sum_{j=1}^{r+1}\|\chi_{r+1,j}(\tau\D)\varphi_r(\tau_n\D)^{-1}\bY^n_r\|_{\mathbb{R}^M}^2\frac{\tau_n}{2j-1}.
\ee
Denote $\mathbf{Z}_r^n=\varphi_r(\tau_n\D)^{-1}\bY_r^n$. If $\D$ is skew-symmetric $\D^T=-\D^T$, by \eqref{d2}, \eqref{eqn_app2}, we have
\begin{align*}
\|\hat{\bY}_r\|_{L^2(I_n)}^2&=\sum_{j=1}^{r+1}\|\chi_{r+1,j}(\tau_n\D)\mathbf{Z}_r^n\|_{\mathbb{R}^M}^2\frac{\tau_n}{2j-1}\\
&=\tau_n\|\varphi_r(\tau_n\D)\mathbf{Z}_r^n\|_{\mathbb{R}^M}^2-\tau_n\left\|\sqrt{\frac{2r}{2r+1}}\frac{r!}{(2r)!}(\tau_n\D)^r\mathbf{Z}_r^n\right\|_{\mathbb{R}^M}^2\\
&\leq\tau_n\|\bY_r^n\|_{\mathbb{R}^M}^2.
\end{align*}
This shows the claim \eqref{k6} when $\D$ is skew-symmetric. 

If $\D$ is symmetric, the eigenvalues of $\D$ are non-positive since $\D+\D^T\le 0$. By Lemma \ref{lem:5.2}, it is easy to show that
\ben
\sum^{r+1}\|\chi_{r+1,j}(\tau_n\D)\mathbf{Z}^n_r\|_{\R^M}^2\frac{\tau_n}{2j-1}-\tau_n\|\varphi_r(\tau_n\D)\mathbf{Z}^n_r\|_{\R^M}^2\le 0,
\een
which yields
\ben
\sum^{r+1}_{j=1}\|\chi_{r+1,j}(\tau_n\D)\varphi_r(\tau_n\D)^{-1}\bY^n_r\|_{\R^M}^2\frac{\tau_n}{2j-1}\le\tau_n\|\bY^n_r\|_{\R^M}^2.
\een
Now it follows from \eqref{k7} that $\|\hat\bY_r\|_{L^2(I_n)}\le\tau_n^{1/2}\|\bY^n_r\|_{\R^M}$. This shows the claim \eqref{k6} when $\D$ is a symmetric matrix.

It follows from \eqref{k6}, \eqref{g7} and \eqref{g3} that
\begin{align*}
\|{\bY}_r\|_{L^2(I_n)}&\leq \tau_n^{1/2}\|\bY_r^n\|_{\mathbb{R}^M}+2\tau_n\|\bR\|_{L^2(I_n)}\\
&\leq \tau_n^{1/2}\|\bY_0\|_{\mathbb{R}^M}+C\tau_n^{1/2}T^{1/2}\|\bR\|_{L^2(0,T)}
\end{align*}
This implies \eqref{stab_L2} easily. Now by the $hp$ inverse estimate,
\begin{align*}
\max_{t_n\leq t\leq t_{n+1}}\|\bY_r\|_{\mathbb{R}^M}\leq C\tau_n^{-1/2}r\|\bY_r\|_{L^2(I_n)}\leq Cr(\|\bY_0\|_{\mathbb{R}^M}+T^{1/2}\|\bR\|_{L^2(0,T)}).
\end{align*}
This shows \eqref{stab_Linf}. This completes the proof.  
\end{proof}

The following theorem which improves the error estimates in Theorem \ref{thm:2.1} can be proved by the argument in Theorem \ref{thm:2.1} by using Theorem \ref{thm:5.1} instead of Lemma \ref{lem:2.2}. Here we omit the details.

\begin{thm}\label{thm:5.2}
Let $\D$ be a symmetric or skew-symmetric matrix. Assume that $\bR\in [H^s(0,T)]^M$, $\bY\in [W^{1+s,\infty}(0,T)]^M$, $s\ge 1$, and $\bY_r\in\bV^r_\tau$ is the solution of the problem \eqref{dp}, we have
\ben
& & \|\bY - \bY_r\|_{L^2(0,T)} \leq C(1+T)\frac{\tau^{\min(r+1,s)}}{r^s}(\|\bY\|_{H^s(0,T)}+\|\D\bY\|_{H^s(0,T)}),\\
& &\max_{0\leq t\leq T}\|\bY - \bY_r\|_{\R^M}\leq C(1+T^{1/2})\frac{\tau^{\min(r+1,s)}}{r^{s-1}}(T^{1/2}\|\bY\|_{W^{s+1,\infty}(0,T)}+\|\bR\|_{H^s(0,T)}),
\een
where the constant $C$ is independent of $\tau,r,\D$, and $\bR$ but may depend on $s$.
\end{thm}

We remark that the first estimate in Theorem \ref{thm:5.2} is optimal both in $\tau$ and $r$.

\section{Numerical examples}
\label{sec_numeric}
In this section, we provide some numerical examples to confirm the theoretical results in this paper.

\begin{exmp}\label{exampe_1} (Dissipative problem)
Let $\Om=(0,1)\times (0,1)$ and $T=1$. We consider the following constant coefficient convection-diffusion problem
\begin{align}\label{2d_convection_diff}
\left\{\begin{aligned}
&u_t+\nabla\cdot(\bm{\beta}u-\epsilon\nabla u)=f \ \ \mbox{in } \Omega\times(0,T),\\
&u(\mathbf{x},0)=u_0(\mathbf{x})\ \ \mbox{in }\Om.
\end{aligned}\right.
\end{align}
The boundary condition is set to be periodic. The source term $f$ is chosen such that the exact solution is $u(\mathbf{x},t)=\exp(-t)\sin(4\pi(x_1-t))\cos(4\pi(x_2-t))$.
 \end{exmp}

We choose $\bm{\beta}=(1,1)^T$ and $\epsilon=1$ in \eqref{2d_convection_diff}. For spatial discretizations, we apply the local discontinuous Galerkin (LDG) method in Cockburn and Shu \cite{CockburnShu1998SINUM} by using purely upwind fluxes for convection terms and alternating fluxes for diffusion terms. For the sake of completeness, we recall the method for solving \eqref{2d_convection_diff} here.

Let $\mathcal{M}$ denote a uniform Cartesian mesh of $\Om$ with $h$ the length of the sides of the elements.
$\mathcal{E}=\mathcal{E}^{\rm side}\cup\mathcal{E}^{\rm bdy}$, where $\mathcal{E}^{\rm side}:=\{e=\partial K \cap \partial K': K, K'\in \mathcal{M}\}$, $\mathcal{E}^{\rm bdy}:=\{e=\partial K \cap \partial \Omega: K\in \mathcal{M}\}$. For any subset $\widehat{\cam}\subset\cam$ and $\widehat{\caE}\subset\caE$, we use the notation
\begin{align*}
(u,v)_{\widehat{\cam}}=\sum_{K \in \widehat{\cam}}(u,v)_{K}, \ \ \langle u,v\rangle_{\widehat{\mathcal{E}}}=\sum_{e \in \widehat{\caE}}\langle u,v\rangle_{e},
\end{align*}
where $(\cdot,\cdot)_K$ and $\la\cdot,\cdot\ra_e$ denote the inner product of $L^2(K)$ and $L^2(e)$, respectively.

For any $e\in \caE$, we fix a unit normal vector $\mathbf{n}_e$ of $e$ with the convention that $\mathbf{n}_e$ is the unit outer normal to $\partial \Omega$ if $e \in \caE^{\rm bdy}$. For any $v\in H^1(\cM):=\{v:v\in H^1(K), K\in\cM\}$, we define the jump operator of $v$ across $e$:
\begin{align*}
[\![v]\!]_e:=v^{-} -v^{+}\ \  \forall e \in \caE^{\rm side},\ \ \ \
[\![v]\!]_e:=v^{-}\ \ \forall e \in \caE^{\rm bdy},
\end{align*}
where $v^{\pm}(\mathbf{x}):=\lim_{\varepsilon\rightarrow 0^+} v(\mathbf{x}\pm\varepsilon \mathbf{n}_e)\ \forall\mathbf{x}\in e$.
For any integer $p\geq0$, we define the finite element space
\begin{align*}
V_h^p:=\{v\in L^2(\Omega):v|_{K}\in Q^p(K), K\in \mathcal{M}\},
\end{align*}
where $Q^p(K)$ denotes the space of polynomials of degree at most $p$ in each variable in $K$.

The semi-discrete problem is to find $(u_h,\mathbf{q}_h)\in [V_h^p]^3$ such that, for all test functions $(v_h,\mathbf{r}_h)\in [V_h^p]^3$,
\begin{align*}
&(\partial_tu_h,v_h)_{\mathcal{M}}+\mathcal{G}(\bm{\beta}u_h,v_h)=\sqrt{\epsilon}\,\left[-(\mathbf{q}_h,\nabla v_h)_{\mathcal{M}}+\langle {\mathbf{q}}^-_h\cdot \mathbf{n},[\![v_h]\!]\rangle_{\mathcal{E}}\right]
+(f,v_h)_{\mathcal{M}},\\
&(\mathbf{q}_h,\mathbf{r}_h)_{\mathcal{M}}=\sqrt{\epsilon}\,\left[-(u_h,\text{div} \mathbf{r}_h)_{\mathcal{M}}+\langle {u}^+_h,[\![\mathbf{r}_h]\!]\cdot \mathbf{n}\rangle_{\mathcal{E}}\right],\\
&u_h(\mathbf{x},0)=(\mathcal{P}_hu_0)(\mathbf{x})\ \ \mbox{in }\Om.
\end{align*}
Here $\mathcal{P}_h:L^2(\Om)\to V_h^p$ is the standard $L^2$ projection operator, and
\begin{align*}
&\mathcal{G}(\bm{\beta}u_h,v_h)=-(\bm{\beta}u_h,\nabla v_h)_{\mathcal{M}}+\langle \check{u}_h\bm{\beta}\cdot \mathbf{n},[\![v_h]\!]\rangle_{\mathcal{E}},\end{align*}
where $\check{u}_h$ is chosen as the upwind flux: $\check{u}_h=u_h^-$ if $\bm{\beta}\cdot \mathbf{n} >0$, $\check{u}=u^+_h$ if $\bm{\beta}\cdot \mathbf{n} <0$. For $e\in \caE^{\rm bdy}$, we use the periodic boundary condition to define ${u}_h^+$.

The optimal $L^2$-norm error estimate of order $p+1$ of the semi-discrete scheme for quasi-uniform Cartesian meshes can be found in Cheng et al \cite[Theorem 2.4]{YaoMC2017}, where it is shown that
$\max_{0\le t\le T}\|u-u_h\|_{L^2(\Omega)}\leq C (1+T)h^{p+1}$.
Therefore, combined with the continuous time Galerkin scheme, we know that the fully discrete scheme has $O(h^{p+1}+\tau^{r+1})$ accuracy in the norm $\|\cdot\|_{L^\infty(0,T;L^2(\Omega))}$ and $O(h^{p+1}+\tau^{2r})$ in the $L^2$ norm at nodes $t=t_n$, $n=1,\cdots,N$.

To test the accuracy at the nodes, we set $\tau=h^{\frac{p+1}{2r}}$ and thus $N=T/\tau=T\beta^{\frac 1r}$, where $\beta=h^{-\frac{p+1}{2}}$. The numerator of the $[r/r]$ Pad\'e approximation $P_r(z)$ has $2k$ complex zeros and $1$ real root if $r=2k+1, k\ge 1$, and $2k$ complex zeros if $r=2k$, $k\ge 1$. Denote by $C(2M)$ the costs of solving the matrix problem $\tau_n\D+\zeta_j\I$ with $\zeta_j$ being complex and $C(M)$ the costs of solving the matrix problem $\tau_n\D+\zeta_j\I$ with $\zeta_j$ real, where $\zeta_j$, $j=1,\cdots, r$, are zeros of $P_r(z)$. Then the computational time in each time step of Algorithm \ref{aa1} is proportion to $C(2M)$ for the parallel computation and proportion to $kC(2M)+(r-2k)C(M)$ for the sequential computation. The wall time of using Algorithm \ref{aa1} using parallel machines is then proportion to $N=T\beta^{\frac 1r}$ which is decreasing in $r$. Thus high order time discretization is preferred for parallel computations. On the other hand, for the sequential computation, the wall time of using Algorithm \ref{aa1} is proportion to $rN=Tr\beta^{\frac 1r}$ which minimizes at $r=\ln\beta$ for $r>0$. This implies that the optimal choice of the order for the sequential computation is $r=\lfloor\ln\beta\rfloor+1$, where $\lfloor a\rfloor$ is the maximum integer strictly less than $a>0$. Table \ref{tab1} shows the error $\|(u-u_h)(\cdot,T)\|_{L^2(\Om)}$ at the terminal time when $r=\lfloor\ln\beta\rfloor+1$. The optimal $(p+1)$-th order is observed which confirms our theoretical results. We observe that the errors of high order schemes are significant smaller than the low order schemes.

To test the accuracy in the $\|\cdot\|_{L^\infty(0,T;L^2(\Om))}$ norm, we set $\tau=h^{\frac{p+1}{r+1}}$ and thus $N=T/\tau=T\gamma^{\frac 1{r+1}}$, where $\gamma=h^{-(p+1)}$. The wall time of using Algorithm \ref{aa1} and Algorithm \ref{aa2} for the parallel computation is proportion to $N=T\gamma^{\frac 1{r+1}}$ which decreases in $r$. On the other hand, for the sequential computation, the wall time of using Algorithm \ref{aa1} and Algorithm \ref{aa3} is $rN=Tr\gamma^{\frac 1{r+1}}$ which is increasing in $r$ if $\ln\gamma\le 4$ and minimizes at $r^*=[-(2-\ln\gamma)+\sqrt{(2-\ln\gamma)^2-4}\,]/2$ if $\ln\gamma\ge 4$. Since $r^*\ge 1$ is equivalent to $\ln\gamma\ge 4$, the optimal choice of the order for minimizing the computation wall time is $r=\max(1,\lfloor r^*\rfloor+1)$. We note that for the sequential computation, Algorithm \ref{aa3} is cheaper than Algorithm \ref{aa2}. Table \ref{tab2} shows the error
\ben
\max_{0\le n\le N-1,1\le k\le 10}\|(u-u_h)(\cdot,t_n+0.1k\tau_n)\|_{L^2(\Om)}
\een
as the approximation of $\|u-u_h\|_{L^\infty(0,T;L^2(\Omega))}$ when $r=\max(1,\lfloor r^*\rfloor+1)$. We again observe the optimal $(p+1)$-th order convergence and that high order methods perform much better than low order methods.

\begin{table}[!ht]\centering
	\caption{Example \ref{exampe_1}: numerical errors of $\|(u-u_h)(\cdot,T)\|_{L^2(\Om)}$ and orders.}\label{tab1}

\begin{tabular}{|c|cc|cc|cc|}
  \hline
 &\multicolumn{2}{|c|}{$p=3$}&\multicolumn{2}{|c|}{$p=4$}&\multicolumn{2}{|c|}{$p=5$}\\\hline
  $h$    & error & order& error & order & error & order  \\ \hline
 $1/4$  & 8.06E-03 & -- & 1.29E-03 & --  & 1.71E-04 &  --  \\
 $1/8$ & 5.64E-04 & 3.84	& 4.40E-05 & 4.88& 2.84E-06 & 5.91  \\
 $1/16$  & 3.55E-05 & 3.99	& 1.37E-06 & 5.01& 4.33E-08 & 6.04\\
 $1/32$  & 2.03E-06 & 4.13	& 4.26E-08 & 5.00& 6.94E-10 & 5.96  \\ \hline
\end{tabular}

\end{table}

\begin{table}[!ht]\centering
	\caption{Example \ref{exampe_1}: numerical errors in $\|\cdot\|_{L^\infty(0,T;L^2(\Om))}$ norm and orders.}\label{tab2}

\begin{tabular}{|c|cc|cc|cc|}
  \hline
 &\multicolumn{2}{|c|}{$p=3$}&\multicolumn{2}{|c|}{$p=4$}&\multicolumn{2}{|c|}{$p=5$}\\\hline
  $h$    & error & order& error & order & error & order  \\ \hline
 $1/4$  & 2.77E-02 & --   &3.99E-03 & -- & 5.81E-04 &  --  \\
 $1/8$ & 2.00E-03 & 3.79	&1.36E-04 & 4.87 & 1.02E-05 & 5.83  \\
 $1/16$ & 1.14E-04 & 4.14	&4.40E-06 & 4.95 & 1.69E-07 & 5.91\\
 $1/32$  & 7.18E-06 & 3.98	&1.33E-07 & 5.05 & 2.73E-09 & 5.95 \\ \hline
 \end{tabular}

\end{table}

\begin{exmp}\label{example4}
(Wave propagation problem) Let $\Omega=(-2,2)\times(-2,2)$ and $T=1$. We consider the following wave equation with discontinuous coefficients
\begin{align}\label{modelproblem}
\left\{\begin{aligned}
&\frac{1}{\rho c^2} \pa_tu= \mbox{\rm div}\, \mathbf{q}+f,\ \  \rho\pa_t\mathbf{q}=\nabla u \ \ \ \ \mbox{in }\Omega\times (0,T),\\
&[\![u]\!]=0,\quad [\![\mathbf{q}\cdot \mathbf{n}]\!]=0\ \ \ \ \text{ on } \Gamma\times (0,T),\\
&u=0\quad \text{ on }\partial\Omega\times(0,T),\\
&u(\mathbf{x},0)=u_0(\mathbf{x}),\ \  \mathbf{q}(\mathbf{x},0)=\mathbf{q}_0(\mathbf{x})\ \ \ \ \mbox{in }\Omega.
\end{aligned}\right.
\end{align}
We assume the interface $\Gamma$ is the union of two closely located ellipses. We take $\Omega_1=\{\mathbf{x}\in\Omega:\frac{(x_1-d_1)^2}{a^2}+\frac{x_2^2}{b^2}< 1 \text{ or }\frac{(x_1-d_2)^2}{a^2}+\frac{x_2^2}{b^2}< 1\}$, which is the union of two disks, and $\Omega_2=\Omega\backslash \bar{\Omega}_1$. Here $d_1=-0.82$, $d_2=0.82$, $a=0.81$, and $b=0.51$. The distance between two ellipses is $0.02$. We consider the wave equation \eqref{modelproblem} with $\rho_1=1/2,\, \rho_2=1$, $c_1=c_2= 1$, and the source $f$ is chosen such that the exact solution is
\begin{align*}
&u(\mathbf{x},t)=\left\{\begin{array}{cc}
\cos(3t)\sin(r_1-1)\sin(r_2-1)\sin(3\pi x_1)\sin(3\pi x_2) & \text{ in } \Omega_1,\\
\\
2\cos(3t)\sin(r_1-1)\sin(r_2-1)\sin(3\pi x_1)\sin(3\pi x_2) & \text{ in } \Omega_2,
\end{array}\right.
\end{align*}
where $r_1=\frac{(x_1-d_1)^2}{a^2}+\frac{x_2^2}{b^2}\quad r_2=\frac{(x_1-d_2)^2}{a^2}+\frac{x_2^2}{b^2}$.
The exact solution $\bfq(\bfx,t)$ is computed by \eqref{modelproblem} with the initial condition $\bfq_0=0$.
\end{exmp}

We use the unfitted finite element method in Chen et al \cite{Chen2021preprint} to discretize the problem in space. Let $\mathcal{M}$ be an induced mesh which is constructed from a Cartesian partition $\mathcal{T}$ of the domain $\Omega$ with possible local refinements and hanging nodes so that the elements are large with respect to both domains $\Om_1,\Om_2$. Let $\mathcal{M}^{\Gamma}:=\{K\in\mathcal{M}:K\cap\Gamma\not=\emptyset\}$ and $\caE=\caE^{\rm side}\cup\caE^{\Gamma}\cup\caE^{\rm bdy}$, where $\caE^\Gamma:=\{\GaK=\Gamma \cap K: K\in \mathcal{M} \}$.

For any $K\in\cM^\Ga$, $i=1,2$, let $K_i=K\cap\Omega_i$ and $K_i^h$ the polygonal approximation of $K_i$ bounded by the sides of $K$ and $\Ga_K^h$ which is the line segment connecting two intersection points of $\Ga_K\cap\pa K$. $K_i^h$ is the union of shape regular triangles $K_{ij}^h$, $1\leq J_i^K\leq 3$, whose sides are the sides of $K_i^h$ and $\Ga_K^h$. We always set $K_{i1}^h$ the element having $\Gamma_K^h$ as one of its sides. From $K_{ij}^h$ we define the curved element $\widetilde K_{ij}^h$ by
\ben
\widetilde K_{i1}^h=(K_i\cap K_{i1}^h)\cup(K_i\backslash\bar K_{i1}^h),\ \ \widetilde K_{ij}^h=K_i\cap K_{ij}^h,\ \ j=2,\cdots,J_i^K.
\een
Then we know that $K$ is the union of curved triangles $\widetilde K_{ij}^h$, $i=1,2, j=1,\cdots,J_i^K$.

For any integers $p,q\ge 1$, the space $P^p(K)$ denotes the space of polynomials of degree at most $p$ in $K$ and $Q^{p,q}(K)$ denotes the space of polynomials of degree at most $p$ for the first variable and $q$ for the second variable in $K$. For any $K\in\cM^\Ga$, we define the interface finite element spaces
\ben
W_p(K)=\{\varphi: \varphi|_{\widetilde K_{ij}^h}\in P^p(\widetilde K_{ij}^h),\ i=1,2,\ j=1,\cdots,J_i^K\},
\een
and $X_p(K)=W_p(K)\cap H^1(K_1\cup K_2)$. Notice that the functions in $X_p(K)$ are conforming in each $K_i, i=1,2$.
Now we define the following unfitted finite element spaces
\ben
& &X_p(\mathcal{M}):=\{v\in H^1(\Om_1\cup\Om_2): v|_K\in X_p(K)\ \ \forall K\in\cam^\Ga,\\
& &\hskip4.8cm v|_K\in Q^p(K)\ \ \forall K\in\cam\backslash\cam^\Ga\},\\
& &\mathbf{W}_p(\mathcal{M}):=\{\bmp: \bmp|_K \in [W_p(K)]^2 \ \ \forall K\in \mathcal{M}^{\Gamma}, \\
& &\hskip4.3cm \bmp|_K\in Q^{p-1,p}(K)\times Q^{p,p-1}(K)\ \ \forall K\in\cam\backslash\cam^\Ga\}.
\een
Let $X_p^0(\cM)=X_p(\cM)\cap H^1_0(\Om_1\cup\Om_2)$, where $H^1_0(\Om_1\cup\Om_2)=\{v\in H^1(\Om_1\cup\Om_2): v=0 \mbox{ on } \pa\Om\}$.

The semi-discrete unfitted finite element method for solving \eqref{modelproblem} is then to find $(u_h,\bfq_h)\in X_p^0(\cM)\times \fespaceq$ such that for all $(\varphi_h,\bmp_h)\in X_p^0(\cM)\times\fespaceq$,
\begin{align}
&\Big(\frac{1}{\rho c^2}\pa_tu_h,\varphi_h\Big)_{\cam}=-(\bfq_h,\nabla \varphi_h)_{\cam}+\langle{\bfq}_h^-\cdot \mathbf{n}, [\![\varphi_h]\!]\rangle_{\mathcal{E}^{\Gamma}}+(f,\varphi_h)_{\cam},\label{ldgsch1}\\
&(\rho\pa_t\bfq_h,\bmp_h)_{\cam}=-(u_h, \text{div}\bmp_h)_{\cam}+\langle {u}_h^+, \lj \bmp_h \rj\cdot \mathbf{n}\rangle_{\mathcal{E}},\label{ldgsch2}\\
&u_h(\mathbf{x},0)=(\mathcal{P}_h u_0)(\mathbf{x}),\ \ \bfq_h(\mathbf{x},0)=(\bm{P}_h\bfq_0)(\mathbf{x})\ \ \mbox{in }\Omega,\label{ldgsch3}
\end{align}
where $\cP_h:L^2(\Om)\to X_p^0(\cM)$ and $\bm{P}_h:[L^2(\Om)]^2\to\fespaceq$ are the standard $L^2$ projection operators.

It is shown in \cite[Theorem 2.2]{Chen2021preprint} that the following energy error of the semi-discrete scheme
\begin{align*}
E_{en}(t):=(\|(u- {u}_h)(\cdot,t)\|_{L^2(\Omega)}^2+\|(\bfq-{\bfq}_h)(\cdot,t)\|_{L^2(\Om)}^2)^{1/2}.
\end{align*}
has $p$-th order convergence.
The semi-discrete problem \eqref{ldgsch1}-\eqref{ldgsch3} is an ODE system which is solved by the continuous time Galerkin method in this paper. By Theorem \ref{thm:2.1} and Theorem \ref{super_conver}, we know that the energy error has $O(h^{p}+\tau^{r+1})$ convergence rate in the norm $\max_{0\le t\le T}E_{en}(t)$ and $O(h^{p}+\tau^{2r})$ in the $E_{en}(t)$ at nodes $t=t_n$, $n=1,\ldots,N$.

Note that these two ellipses are close but not tangent. To resolve the interface $\Gamma$ well, we locally refine the mesh near the interface such that the interface deviation $\eta_K\leq\eta_0=0.05$ for all $K\in \mathcal{M}^\Gamma$. For the concept of the interface deviation we refer to \cite[Definition 2.2]{Chen2021preprint}, see also Chen et al \cite[Definition 2.2]{ChenLi}. As an illustration, we show the computational mesh for $h=1/4$ in Figure \ref{fig_mesh_exmp_two_circles}. 


In this example, we test the accuracy of the error in the $\|\cdot\|_{L^\infty(0,T;L^2(\Om))}$ norm and $E_{en}(T)$. As in Example 1, to test the accuracy at nodes, we set $\tau=\frac{p}{2r}$, thus the wall time of using Algorithm \ref{aa1} for the sequential computation is proportion to $Tr\nu^{\frac 1r}$ which minimizes at $r=\ln \nu$ for $r>0$, where $\nu=h^{-\frac{p}{2}}$. Table \ref{tab3} shows the error $E_{en}(T)$ at the terminal time when $r=\lfloor\ln\nu\rfloor+1$.

To test the accuracy in the $\|\cdot\|_{L^\infty(0,T;L^2(\Om))}$ norm, we set $\tau=h^{\frac p{r+1}}$ and thus $N=T/\tau=T\mu^{\frac 1{r+1}}$, where $\mu=h^{-p}$. As in Example 1, the wall time of using Algorithm \ref{aa1} and Algorithm \ref{aa2} for the parallel computation is proportion to $N=T\mu^{\frac 1{r+1}}$ which decreases in $r$. However, since the wave equation is not dissipative, we cannot use Algorithm \ref{aa3} for the sequential computation. In this case, the wall time of using Algorithm \ref{aa1} and Algorithm \ref{aa2} for the sequential computation is proportion to $rN+r^2N=Tr(r+1)\mu^{\frac 1{r+1}}$ which is increasing if $\ln\mu\le 6$ and minimizes at $r^{**}=[-(3-\ln\mu)+\sqrt{(3-\ln\mu)^2-8}\,]/4$ if $\ln\mu\ge 6$. Since $r^{**}\ge 1$ is equivalent to $\ln\mu\ge 6$, the optimal choice of the order for minimizing the computation wall time is $r=\max(1,\lfloor r^{**}\rfloor+1)$. Table \ref{tab3} shows the error
\ben
\max_{0\le n\le N-1,1\le k\le 10}E_{en}(t_n+0.1k\tau_n)
\een
as the approximation of $\max_{0\le t\le T}E_{en}(t)$ when $r=\max(1,\lfloor r^{**}\rfloor+1)$. We clearly observe the optimal $p$-th order convergence and the superior performance of high order methods from Tables \ref{tab3}-\ref{tab4}.


\begin{figure}[!ht]
\begin{minipage}[c]{0.45\textwidth}
\includegraphics[width=1.1\textwidth, height = 0.9\textwidth]{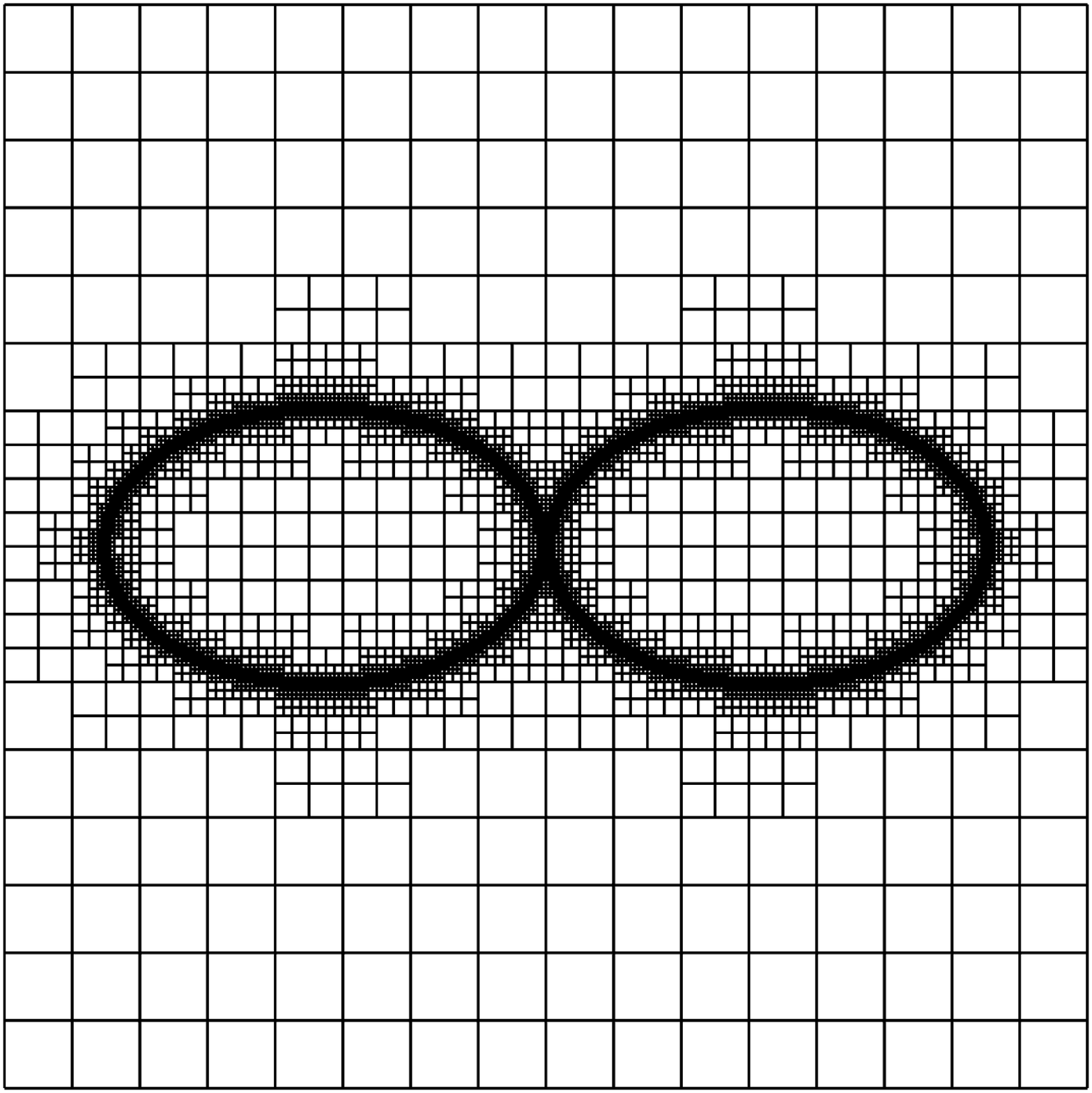}
\end{minipage}
\hspace{-0.5cm}
\begin{minipage}[c]{0.45\textwidth}
\includegraphics[width=0.9\textwidth, height = 0.9\textwidth]{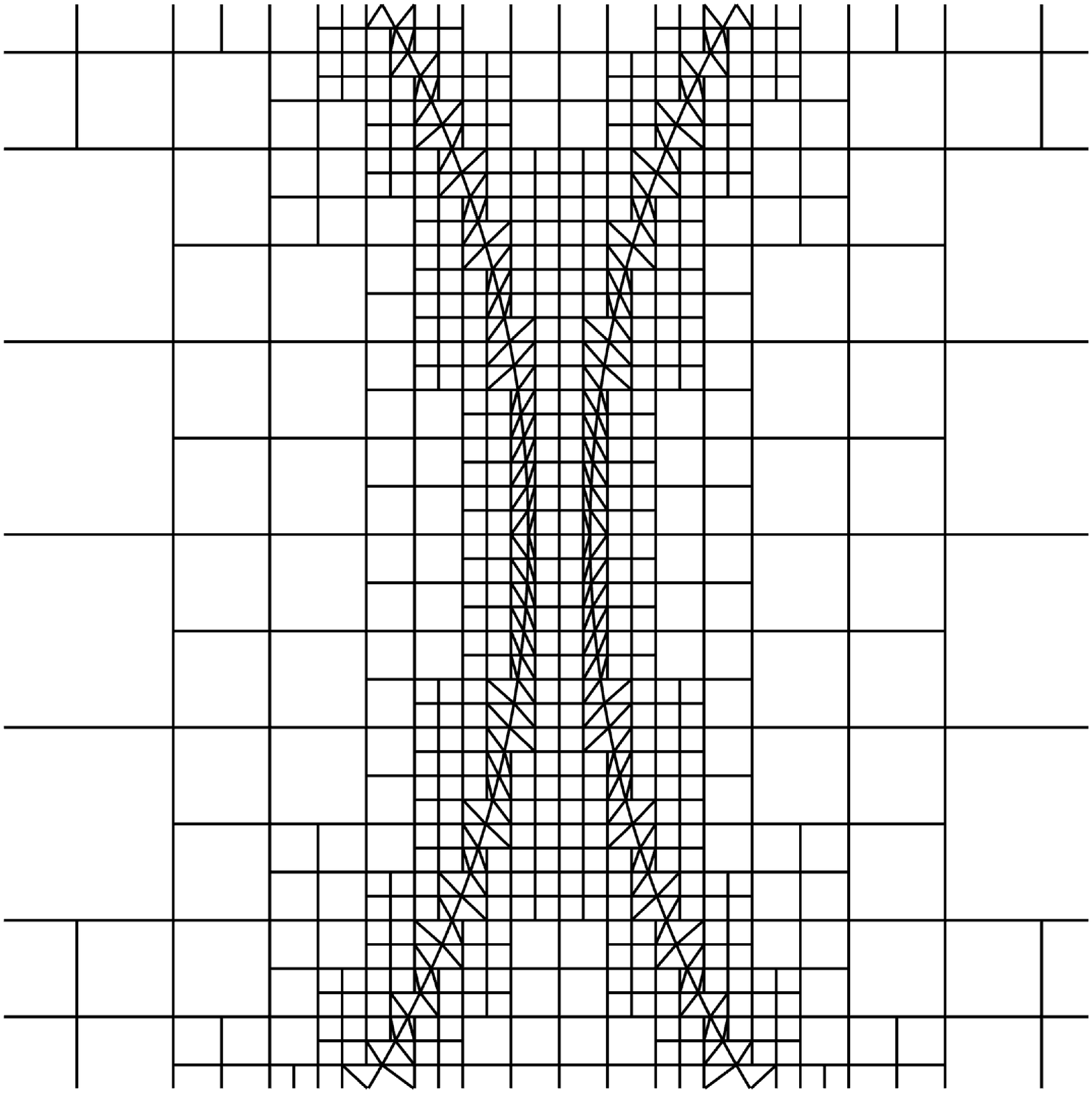}
\end{minipage}
\caption{\label{fig_mesh_exmp_two_circles}  Illustration of the computational domain and the mesh (left) and the corresponding zoomed local mesh (right) with $h=1/4$ in Example \ref{example4}.}
\end{figure}

\begin{table}[!ht]\centering
	\caption{Example \ref{example4}: numerical errors of $E_{en}(T)$ and orders.}\label{tab3}

\begin{tabular}{|c|cc|cc|cc|}
  \hline
 &\multicolumn{2}{|c|}{$p=3$}&\multicolumn{2}{|c|}{$p=4$}&\multicolumn{2}{|c|}{$p=5$}\\\hline
  $h$   & error & order & error & order & error & order  \\ \hline
 $1/4$ & 4.22E-01 & --	& 2.21E-01 & --& 1.25E-01 & -- \\
 $1/8$  &8.97E-02 & 2.24 & 2.64E-02 & 3.06& 6.20E-03 & 4.33\\
 $1/16$  & 1.32E-02 & 2.76& 1.81E-03 & 3.87& 3.82E-05 & 4.96  \\
  $1/32$  & 1.66E-03 & 2.99& 1.13E-04 & 4.00& 1.19E-06 & 5.00 \\ \hline
 \end{tabular}
\end{table}

%
%

%
%

\begin{table}[!ht]\centering
	\caption{Example \ref{example4}: numerical errors of $\max_{0\le t\le T}E_{en}(t)$ and orders.}\label{tab4}

\begin{tabular}{|c|cc|cc|cc|}
  \hline
 &\multicolumn{2}{|c|}{$p=3$}&\multicolumn{2}{|c|}{$p=4$}&\multicolumn{2}{|c|}{$p=5$}\\\hline
  $h$  & error & order & error & order & error & order  \\ \hline
  $1/4$  &5.23E-01   & --	  &2.95E-01  & --  & 1.69E-01 & --	 \\
 $1/8$  & 1.26E-01  &2.05 	  &3.70E-02   & 3.00	  &  8.87E-03 & 4.25 \\
 $1/16$ & 1.90E-02  &2.72	 &2.58E-03   & 3.84	 & 2.93E-04 & 4.92 \\
  $1/32$ & 2.45E-03  &2.96 	  & 1.66E-04   & 3.95 	&9.21E-06 & 4.99     \\ \hline
\end{tabular}

\end{table}

\end{document}